\documentclass[letterpaper]{amsart}%
\usepackage{ifpdf}
\ifpdf																															
 \usepackage[implicit,naturalnames,hyperfootnotes=false,
						 final]{hyperref}
 \usepackage[pdftex]{color}
\else
 \usepackage{color}																									
 \gdef\texorpdfstring#1#2{#1}
\fi
\usepackage{lmodern}																								
\usepackage{xcolor}																									
\usepackage{amsmath}																								
\usepackage{amssymb}																								
\usepackage{amsthm}																									
\usepackage{bbm}																										
\usepackage{amsfonts}																								
\usepackage{mathrsfs}																								
\usepackage{calligra}																								
\usepackage{dsfont}																									
\usepackage{nicefrac}																								
\usepackage{esint}																									
\usepackage{mathtools}																							
\usepackage{ifthen}																									
\usepackage{xparse}																									
\usepackage{xspace}																									
\usepackage{calc}																										
\usepackage{partial}																
\gdef\mytheorem{\relax}\gdef\mydefinition{\theoremstyle{definition}}
  \let\oldnewtheorem\newtheorem%
  \gdef\newtheorem#1{\expandafter\NewDocumentCommand\csname #1\endcsname{o}{\csname old#1\endcsname}\expandafter\gdef\csname end#1\endcsname{\csname endold#1\endcsname}
	\oldnewtheorem{old#1}}
	\gdef\mytheoremcounter{oldtheorem}
\mytheorem
\newtheorem{theorem}{Theorem}[section]
\newtheorem{corollary}[\mytheoremcounter]{Corollary}
\newtheorem{lemma}[\mytheoremcounter]{Lemma}
\newtheorem{proposition}[\mytheoremcounter]{Proposition}
\mydefinition
\newtheorem{definition}[\mytheoremcounter]{Definition}
\newtheorem{defiundlemma}[\mytheoremcounter]{Definition and Lemma}
\newtheorem{construction}[\mytheoremcounter]{Construction}
\theoremstyle{remark}
\newtheorem{remark}[\mytheoremcounter]{Remark}
\gdef\ii{I}\gdef\ij{J}\gdef\ik{K}\gdef\il{L}
\gdef\oi{i}\gdef\oj{j}\gdef\ok{k}\gdef\ol{l}
\gdef\schwarzs{\mathcal S}\let\schws\schwarzs
\gdef\schwarzg{\outg[\schwarzs]}
\gdef\euclideane{e}\let\euce\euclideane
\gdef\eukg{\outg[\euclideane\hspace{.05em}]}
\gdef\rad{r}
\gdef\time{t}
\gdef\atime{\tau}
\gdef\volume#1{\left|#1\right|}
\gdef\lieD#1#2{\mathfrak L_{#1}#2}
\undef\c\NewDocumentCommand\c{G{c}}{\DoATensor{#1}}
\NewDocumentCommand\Cof{G{C}d<>o}
{\IfValueTF{#3}{\Cof{#1}<\IfValueTF{#2}{#2,}\relax#3>}{#1\IfValueTF{#2}{(#2)}\relax}}%
\gdef\rnu{\DoATensor[\!]u}
\gdef\tnu{\DoATensor[\hspace{-.05em}] w}
\gdef\lapse{\DoATensor[\!]u}
\gdef\rbeta{\DoATensor[\!]\beta}
\gdef\centerz{\DoATensor<\hspace{-.05em}>[\!]{\boldsymbol z}}
\let\g\metric
\let\ric\ricci
\gdef\scalar{\DoATensor{\newmathcal S}}
\let\zFund\k
\let\zFundtrf\ktrf
\let\H\mc
\let\oldnu\nu
\gdef\normal{\DoATensor[\!]{\oldnu}}
\let\nu\normal
\gdef\laplace{\DoATensor[\!]\Delta}
\gdef\div{\DoATensor[\!]{\text{div}}}
\gdef\tr{\DoATensor[\!]{\text{tr}}}
\gdef\levi{\DoATensor<\!>\nabla}
\gdef\mug{\DoATensor[\!]{\mu}}
\gdef\jacobiext{\DoATensor<\hspace{-.05em}>[\hspace{-.05em}]{\textrm L}}
\NewDocumentCommand\trzd{D<>{}O{}mm}
{\mathop{{#3\DoATensor<#1>[#2]\odot{#4}}}}
\NewDocumentCommand\trtr{D<>{}O{}mm}
{\mathop{{\ifthenelse{\equal{#3}{#4}}
 {{\left|#3\right|_{\g<#1>[#2]}^2}}
 {\tr<#1>[#2](\trzd<#1>[#2]{#3}{#4})}}}}
\let\outsymbol\overline
\let\outnu\outnormal
\gdef\radialdirection{\outsymbol\mu}
\NewDocumentCommand\outc{G{c}}{\DoAoutTensor{#1}}\let\oc\outc
\gdef\outdelta{\DoATensor\delta}
\gdef\outve{\DoATensor\ve}
\gdef\ralpha{\DoAoutTensor[\hspace{-.05em}]\alpha}
\gdef\outbeta{\DoAoutTensor[\!]\beta}
\gdef\outcenter{\DoAoutTensor{\boldsymbol z}}
\gdef\momentum{\DoAoutTensor[\!]\Pi}
\gdef\momenden{\DoAoutTensor[\!]J}
\gdef\enden{\DoAoutTensor[\hspace{-.05em}]\varrho}
\gdef\outlevi{\DoAoutTensor\nabla}
\let\outric\outricci
\let\outg\outmetric
\let\outzFund\outk
\NewDocumentCommand\troutzFund{D<>{}O{}}{\tr<#1>[#2]\hspace{.05em}\outzFund[#2]}
\NewDocumentCommand\outzFundnu{D<>{}O{}}{\mathop{{\outzFund[#2]_{\nu<#1>[#2]}}}}
\gdef\outmc{\DoAoutTensor{\mathcal H}}
\let\outH\outmc
\let\outsc\outscalar
\gdef\outdiv{\DoAoutTensor[\!]{\text{div}}}
\gdef\outtr{\DoAoutTensor[\!]{\text{tr}}}
\NewDocumentCommand\outtrzd{D<>{}O{}mm}
{\mathop{{#3\DoAoutTensor<#1>[#2]\odot{#4}}}}
\NewDocumentCommand\outtrtr{D<>{}O{}mm}
{\mathop{{\ifthenelse{\equal{#3}{#4}}
 {{\left|#3\right|_{\outg<#1>[#2]}^2}}
 {\mathop{{\outtr<#1>[#2](\outtrzd<#1>[#2]{#3}{#4})}}}}}}
\let\aoutsymbol\overline
\gdef\arnu{\DoATensor[\hspace{-.05em}] w}
\gdef\atbeta{\DoATensor[\hspace{-.05em}]\gamma}
\gdef\amomentum{\DoAaoutTensor[\!]\Pi}
\let\aoutric\aoutricci
\let\aoutg\aoutmetric
\let\aoutzFund\aoutk
\gdef\aoutdiv{\DoAaoutTensor[\!\!]{\text{div}}}
\gdef\aouttr{\DoAaoutTensor[\!]{\text{tr}}}
\let\unisymbol\widehat
\let\uniric\uniricci
\let\unig\unimetric
\let\unisc\uniscalar
\gdef\outM{\DoAoutTensor[\hspace{-.025em}]{M}}
\gdef\aoutM{\DoAaoutTensor[\hspace{-.025em}]{M}}
\gdef\uniM{\DoAuniTensor[\!]{M}}
\gdef\M{\DoATensor<\hspace{-.05em}>[\hspace{-.05em}]\Sigma}
\NewDocumentCommand\tracefree{m}{\,\,\,\phantomas{}{\!\,^{\!\,^\circ}}\!\!\!#1}
\gdef\Hradius{\sigma}
\gdef\rradius{\mathcal r}
\gdef\impuls{\DoAoutTensor{\textrm P}}
\gdef\aimpuls{\DoATensor{\textrm P}}
\undef\d
\NewDocumentCommand\d{s}{\IfBooleanTF{#1}\relax{\mathop{}\!}\mathrm d}%
\DeclareMathOperator\id{id}
\DeclareMathOperator\lin{lin}
\DeclareMathOperator\graph{graph}
\newcommand\R{\mathds{R}}																						
\newcommand\X{\mathfrak{X}}																					
\newcommand\Lp{\textrm L}
\newcommand\Wkp{\textrm W}
\newcommand\Hk{\textrm H}
\newcommand\Ck{\mathcal C}
\newcommand\sphere{\mathcal S}																			
\newcommand\ve{\varepsilon}																					
\newcommand\Haus{\mathcal{H}}																				
\usepackage{urwchancal}																							
\DeclareMathAlphabet{\mathcal}{OT1}{pzc}{m}{n}
\let\newmathcal\mathcal
\usepackage{auto-eqlabel}																						%
\makeatletter
\NewDocumentCommand\DoATensor{D<>{}O{}md<>os}
{\DoATensor@start{#3}
 {\IfValueTF{#4}
  {\IfValueTF{#5}{\TensorIndex*[#4#1]<#5#2>{#3}}
	 {\TensorIndex*[#4#1]{#3}}}
	{\IfValueTF{#5}{\TensorIndex*<#5#2>{#3}}
	 {\TensorIndex*{#3}}}}}
\NewDocumentCommand\DoAoutTensor{D<>{}O{}md<>os}
{\DoATensor@start{#3}
 {\IfValueTF{#4}
  {\IfValueTF{#5}{\TensorIndex*[#4#1]<#5#2>{\outsymbol{#3}}}
	 {\TensorIndex*[#4#1]{\outsymbol{#3}}}}
	{\IfValueTF{#5}{\TensorIndex*<#5#2>{\outsymbol{#3}}}
	 {\TensorIndex*{\outsymbol{#3}}}}}}
\NewDocumentCommand\DoAaoutTensor{D<>{}O{}md<>os}
{\DoATensor@start{#3}
 {\IfValueTF{#4}
  {\IfValueTF{#5}{\TensorIndex*[#4#1]<#5#2>{\aoutsymbol{#3}}}
	 {\TensorIndex*[#4#1]{\aoutsymbol{#3}}}}
	{\IfValueTF{#5}{\TensorIndex*<#5#2>{\aoutsymbol{#3}}}
	 {\TensorIndex*{\aoutsymbol{#3}}}}}}
\NewDocumentCommand\DoAuniTensor{D<>{}O{}md<>os}
{\DoATensor@start{#3}
 {\IfValueTF{#4}
  {\IfValueTF{#5}{\TensorIndex*[#4#1]<#5#2>{\unisymbol{#3}}}
	 {\TensorIndex*[#4#1]{\unisymbol{#3}}}}
	{\IfValueTF{#5}{\TensorIndex*<#5#2>{\unisymbol{#3}}}
	 {\TensorIndex*{\unisymbol{#3}}}}}}
\NewDocumentCommand\DoATensor@start{smm}
{\DoATensor@checkforsubindex{#1}{#2}{#3}}
\NewDocumentCommand\DoATensor@checkforsubindex{mmmt_}
{\IfBooleanTF{#4}{\DoATensor@index{#2}{#3}_}
 {\DoATensor@checkforsupindex{#1}{#2}{#3}}}
\NewDocumentCommand\DoATensor@checkforsupindex{mmmt^}
{\IfBooleanTF{#4}{\DoATensor@index{#2}{#3}^}
 {\IfBooleanTF{#1}{\mathop{#3}}{#3}}}
\NewDocumentCommand\DoATensor@index{mmmm}
{\DoATensor@start*{#1}{#2\vphantom{#1}#3{#4}}}
\NewDocumentCommand\TensorIndex{ssod<>msod<>}
{\IfBooleanTF{#1}
 {\TensorIndex[#3]<#4>{#5}*[#7]<#8>}
 {\IfBooleanTF{#2}
  {\mathop{\TensorIndex[#3]<#4>{#5}*[#7]<#8>}}
  {\IfBooleanTF{#6}
   {\Tensor@Indexparse@@left{\IfValueTF{#3}{_{#3}}\relax}{\IfValueTF{#4}{^{#4}}\relax}{#5}%
    #5
    \IfValueTF{#7}{_#7}\relax\IfValueTF{#8}{^#8}\relax}
   {\IfValueTF{#7}{\mathop{\TensorIndex[#3]<#4>{#5}*[#7]<#8>}}
    {\IfValueTF{#8}{\mathop{\TensorIndex[#3]<#4>{#5}*[#7]<#8>}}
	   {\IfValueTF{#3}{\mathop{}\TensorIndex[#3]<#4>{#5}*[#7]<#8>}
	    {\IfValueTF{#4}{\mathop{}\TensorIndex[#3]<#4>{#5}*[#7]<#8>}{\mathop{#5}}}}}}}}}
\newdimen\Tensor@widthmax
\newdimen\Tensor@widthfirst
\newdimen\Tensor@widthsecond
\NewDocumentCommand\Tensor@Indexparse@@left{smmm}
{\IfBooleanTF{#1}
 {\setlength\Tensor@widthfirst{\widthof{\ensuremath{\vphantom{#4}#2}}}%
  \setlength\Tensor@widthsecond{\widthof{\ensuremath{\vphantom{#4}#3}}}%
  \setlength\Tensor@widthmax{\maxof\Tensor@widthfirst\Tensor@widthsecond}
  \kern\Tensor@widthmax
	\kern-\Tensor@widthfirst%
	\vphantom{#4}#2
	\kern-\Tensor@widthsecond%
	\vphantom{#4}#3
  \kern-\scriptspace}
 {\ifthenelse{\equal{#2}\relax}
  {\ifthenelse{\equal{#3}\relax}\relax
	 {\Tensor@Indexparse@@left*{#2}{#3}{#4}}}
	{\Tensor@Indexparse@@left*{#2}{#3}{#4}}}}
\makeatother
\usepackage{bracket-hack}
\newcommand\phantomas[3][c]{
	\ifmmode
		\makebox[\widthof{$#2$}][#1]{$#3$}
	\else
		\makebox[\widthof{#2}][#1]{#3}
	\fi
}
\title[Time evolution of ADM and CMC center of mass]{Time evolution of ADM and CMC\\center of mass in General Relativity}%
 \author[Christopher Nerz]{Christopher Nerz}
 \address{Mathematisches Institut\\Universit\"at T\"ubingen\\Auf der Morgenstelle 10\\72076 T\"ubingen, Germany}
 \email{christopher.nerz@math.uni-tuebingen.de}
 \date\today
\ifpdf
	\hypersetup{pdftitle		= {Time evolution of ADM and CMC center of mass},
							pdfauthor		= {Christopher Nerz},
							pdfsubject	= {Analysis of the evolution of the CMC foliation and center of mass in time under the Einstein equations},
							pdfkeywords	= {Evolution under the Einstein equations, constant mean curvature foliation, ADM center of mass, CMC center of mass, center of mass}}\fi
\begin{document}
\begin{abstract}\noindent 
It is shown by several authors going back to Huisken-Yau that asymptotically Schwarzschildean time-slices possess a unique foliation by stable constant mean curvature (CMC) spheres defining the so-called CMC center of mass. We analyze how the leaves of this foliation evolve in time under the Einstein equations. More precisely, we prove that, asymptotically, their time evolution is a translation induced by the quotient of their linear momentum $\impuls$ and mass $m$, as to be expected from the corresponding Newtonian setting. In particular, the definitions of mass and linear momentum defined by Arnowitt-Deser-Misner (ADM) are compatible with the interpretation of the CMC foliation as the center of mass of the time-slice by Huisken-Yau. Furthermore, we prove that the coordinate version of the center of mass by Arnowitt-Deser-Misner and the coordinate version of the CMC center of mass coincide -- without additional conditions on the scalar curvature. This is even true in the sense of existence, i.\,e.\ if one of the two exists then so does the other.
\end{abstract}
\maketitle
\let\sc\scalar%

\section{Introduction and general considerations}
It is well known, that the \emph{center of mass} $\vec z$ of any isolated Newtonian gravitating system evolves in time via
\begin{align*} \partial[t]@{\vec z} = \frac{\vec P}m, \end{align*}
where $\vec P$ denotes the \emph{linear momentum} and $m>0$ the \emph{mass} of the entire system. It is thus natural to ask whether this is also true when we consider isolated systems in general relativity. For such systems, there is no obvious way to define physical quantities such as mass, linear momentum, and center of mass. For relativistic (total) mass $m$ and relativistic (total) linear momentum $\impuls$, the definitions given by Arnowitt-Deser-Misner (ADM) \cite{arnowitt1961coordinate} are well-established in the literature. However, the contemporary literature knows several definitions of the relativistic center of mass of an isolated system in general relativity.\smallskip

Several authors suggest to define the center of mass of an isolated system as a foliation near infinity of the corresponding Riemannian manifold with total mass $m>0$. Following \cite{cederbaumnerz2013_examples}, we will call such definitions \emph{abstract} to contrast what we call \emph{coordinate} definitions of center of mass, see below. The first such definition was given by Huisken-Yau \cite{huisken_yau_foliation}, who defined the (CMC) center of mass to be the unique foliation near infinity by closed, stable surfaces with constant mean curvature. This was motivated by the idea of using CMC surfaces in this setting by Christodoulou-Yau \cite{christodoulou71some}. Later, Lamm-Metzger-Schulze \cite{lamm2011foliationsbywillmore} used a unique foliation by spheres of Willmore type and (in the static case) Cederbaum \cite[Cor.~3.3.4,\;Thm.~5.3.3]{Cederbaum_newtonian_limit} used level-sets of the lapse function.\smallskip

There are several \emph{coordinate} centers of mass, e.\,g.\ the one suggested by Arnowitt-Deser-Misner \cite{arnowitt1961coordinate}, by Huisken-Yau \cite{huisken_yau_foliation}, by Corvino-Schoen \cite{corvino2006asymptotics} and by Huang \cite{Lan_Hsuan_Huang__Foliations_by_Stable_Spheres_with_Constant_Mean_Curvature} (based on an idea by Schoen). Here, we call them \lq coordinate\rq\ centers of mass as they are defined as a (three-dimensional) vector $\centerz\in\R^3$ which depends on the choice of coordinates near infinity -- at least a-priori. It is well-known that the coordinate CMC center of mass coincides with the coordinate ADM center of mass if the scalar curvature is asymptotically antisymmetric, see \cite{Lan_Hsuan_Huang__Foliations_by_Stable_Spheres_with_Constant_Mean_Curvature} and \cite{metzger_eichmair_2012_unique}. In the context of static isolated systems, Cederbaum \cite[Def.~4.3.1]{Cederbaum_newtonian_limit} defined a \lq pseudo-Newtonian\rq\ (quasi-local and total) coordinate center of mass and proved that it coincides with the coordinate CMC and ADM centers of mass \cite[Thm.~4.3.5]{Cederbaum_newtonian_limit}. Furthermore, she showed that the coordinate pseudo-Newtonian (and thus the coordinate CMC and coordinate ADM) center of mass converges to the Newtonian one in the Newtonian limit $c\to\infty$. The coordinate CMC center of mass is generally well-defined for static isolated systems as the scalar curvature vanishes outside a compact set, see also \cite{cederbaumnerz2013_examples}.\smallskip

Chen-Wang-Yau recently suggested a completely new definition of (quasi-local and total) center of mass, which is given by optimal isometric embeddings into Minkowsky spacetime \cite[Def.~3.2]{chen_wang_yau__Quasilocal_angular_momentum_and_center_of_mass_in_gr}. Additionally, their center of mass fulfills $\spartial[t]@\centerz=\nicefrac\impuls m$ \cite[2.~Theorem]{chen_wang_yau__Quasilocal_angular_momentum_and_center_of_mass_in_gr}. To the best knowledge of the author, they give the first rigorous proof of $\spartial[t]@\centerz=\nicefrac\impuls m$ for \emph{any} definition of center of mass in the setting of isolated systems in general relativity.\medskip\pagebreak[1]

In this paper, we focus mainly on the abstract CMC definition of center of mass. Thus, the central object is the unique foliation $\{\M<\Hradius>\}_{\Hradius>\Hradius_0}$ near infinity of an asymptotically Schwarzschildean three-dimensional Riemannian manifold $(\outM,\outg*)$ by stable spheres $\M<\Hradius>$ with constant mean curvature. These \emph{CMC spheres} can be indexed in different ways, e.\,g.\ by the area $\vert\M<\Hradius>\vert=4\pi\Hradius^2$. We use the mean curvature, i.\,e.\ the surface $\M<\Hradius>$ is the unique element of the foliation with mean curvature $\H<\Hradius>\equiv\nicefrac{-2}\Hradius+\nicefrac{4m}{\Hradius^2}$. Existence and uniqueness of such a foliation was first proven by Huisken-Yau \cite{huisken_yau_foliation}. Metzger \cite{metzger2007foliations}, Huang \cite{Lan_Hsuan_Huang__Foliations_by_Stable_Spheres_with_Constant_Mean_Curvature}, and Eichmair-Metzger \cite{metzger_eichmair_2012_unique} subsequently weakened the decay and regularity assumptions on the metric $\outg*$. Furthermore, Eichmair-Metzger \cite{metzger_eichmair_2012_unique} proved that the foliation exists for arbitrary dimension $\dim\outM\ge3$. The uniqueness results were generalized by Qing-Tian \cite{qing2007uniqueness}, Metzger \cite{metzger2007foliations}, Huang \cite{Lan_Hsuan_Huang__Foliations_by_Stable_Spheres_with_Constant_Mean_Curvature}. Brendle-Eichmair \cite{brendle2013large} proved that uniqueness is only valid for the whole foliation $\{\M<\Hradius>\}_{\Hradius>\Hradius_0}$, not for a single leaf $\M<\Hradius>$. \smallskip

Huisken-Yau \cite{huisken_yau_foliation} defined what we will call following \cite{cederbaumnerz2013_examples} the coordinate CMC center of mass $\centerz<\Hradius>\in\R^3$ to be the limit of the \emph{Euclidean coordinate center} $\centerz<\Hradius>$ of the leaf $\M<\Hradius>$ as $\Hradius\to\infty$, i.\,e.
\begin{equation}
 \centerz_{\text{CMC}} := \lim_{\Hradius\to\infty}\centerz<\Hradius> \qquad\text{with}\qquad
 \centerz<\Hradius> := \int_{\M<\Hradius>} \outsymbol x_i \d\Haus^2, \labeleq{Euclidean_Center_of_mass}
\end{equation}
where $\outsymbol x$ denotes the chosen coordinates in which the assumed asymptotic decay conditions on the metric $\outg*$ are satisfied and where $\Haus^2$ is the measure induced on $\M<\Hradius>$ by the Euclidean metric ${{\eukg}}$. The coordinate ADM center of mass $\centerz_{\text{ADM}}$ is defined by the limit of a surface integral (on the coordinate sphere $\{\vert\outsymbol x\vert=\rradius\}$ for $\rradius\to\infty$) -- comparable to the ADM mass and ADM linear momentum, see \cite{arnowitt1961coordinate}. \smallskip

The first main result of this paper is that $\spartial[t]@\centerz=\nicefrac\impuls m$ holds for the abstract CMC, the coordinate CMC, and the coordinate ADM center of mass and we prove this in three versions:\nopagebreak
\begin{description}
\item[spacetime version (temporal foliation), abstract]$\,\!$\\
Let $\Phi:I\times\outM\to\uniM$ be an asymptotically flat foliation of spacetime $(\uniM,\unig*)$ by asymptotically Schwarzschildean spacelike hypersurfaces ${\outM[t]}:=\outsymbol\Phi(t,\outM)$. Let $\{\M<\Hradius>[t]\}$ denote the unique foliation of $\outM[t]$ near infinity by stable, closed hypersurfaces ${\M<\Hradius>[t]}$ with constant mean curvature $\H<\Hradius>[t]\equiv\nicefrac{-2}\Hradius+\nicefrac{4m}{\Hradius^2}$. We show in Theorem \ref{evolution_of_the_leaves} that the time-evolution of ${\M<\Hradius>[t]}$ is asymptotically (with respect to $\Hradius$) characterized by a quantity $\aimpuls<\Hradius>[t]$ and by the total mass $\DoATensor m[t]$ of ${\outM[t]}$, where $\aimpuls<\Hradius>[t]$ has to be understood as a pseudo quasi-local linear momentum, see Remark \ref{definition_of_linear_momentum}. If the momentum-density $\momenden[t]:=\outdiv[t](\outH[t]\,\outg[t]*-\outzFund[t]*)$ on ${\outM[t]}$ decays \lq fast enough\rq, then $\aimpuls<\Hradius>[t]$ converge to the linear momentum $\impuls[t]$ of $\outM[t]$ as $\Hradius\to\infty$. In this sense, $\spartial[t]@\centerz=\nicefrac\impuls m$ holds abstractly for the abstract CMC center of mass $\centerz$. \pagebreak[1]
\item[spacetime version (temporal foliation), coordinate]$\,\!$\\
Let $\Phi:I\times\outM\to\uniM$ be an asymptotically flat foliation of spacetime $(\uniM,\unig*)$ by asymptotically Schwarz\-schildean, spacelike hypersurfaces ${\outM[t]}:=\outsymbol\Phi(t,\outM)$. Assume the coordinate CMC center of mass $\centerz[t_0]_{\text{CMC}}$ is well-defined at a time $t_0\in I$ and that $\momenden[t]:=\outdiv[t](\outH[t]\,\outg[t]*-\outzFund[t]*)$ on ${\outM[t]}$ decays \lq fast enough\rq\ for all $t\in I$. We prove in Corollary \ref{evolution_of_the_center}, that under these assumptions the coordinate CMC center of mass $\centerz[t]_{\text{CMC}}$ is well-defined for all times $t\in I$ and satisfies ${\spartial[t]@{\centerz[t]_{\text{CMC}}}}=\nicefrac{\impuls[t]}m$.
\item[initial data set version]$\,\!$\\
There is a well-defined version of the main result for abstract asymptotically Schwarz\-schildean initial data sets (Theorem \ref{inf_evolution_of_the_leaves}).
\end{description}\smallskip

The second main result of this paper states that the coordinate ADM and coordinate CMC center of mass coincides for asymptotically Schwarzschildean manifolds $(\outM,\outg*)$ not fulfilling the assumption of the asymptotic antisymmetry of the scalar curvature: The coordinate CMC center of mass is well-defined if and only if the coordinate ADM center of mass is well-defined; moreover they coincide whenever they are well-defined. In particular, $\spartial[t]@\centerz=\nicefrac\impuls m$ holds also for the ADM center of mass, if it is well-defined and if $\momenden$ decays fast enough. We note that \cite{Lan_Hsuan_Huang__Foliations_by_Stable_Spheres_with_Constant_Mean_Curvature} and \cite{metzger_eichmair_2012_unique} assume weaker decay conditions on the metric $\outg*$ itself than we do, but an need asymptotic antisymmetry assumption on the scalar curvature. In particular, the coordinate CMC center of mass is always defined in their setting \cite[Lemma~F.1]{metzger_eichmair_2012_unique}.


\begin{remark}[The asymptotically flat setting]\label{Asymptotically_flat_setting}
It should be noted that all three version of the first result explained above can be carried over to the setting of an asymptotically flat manifold - see Remark \ref{Asymptotically_flat_setting_2}. This can easily be seen by simply replacing Theorem~\ref{existence}, Proposition~\ref{leaves_are_almost_concentric}, and Lemma~\ref{the_lowest_eigenvalues_of_the_stability_operator} by the newer result \cite[Thm~3.1, Prop.~2.4, and Prop.~2.7]{nerz2014CMCfoliation}, respectively.
\end{remark}

\textbf{Acknowledgment.}
The author wishes to express gratitude to Gerhard Huisken for suggesting this topic, many inspiring discussions and ongoing supervision. Further thanks is owed to Katharina Radermacher for proof-reading. Finally, this paper would not have attained its current form and clarity without the useful suggestions by Carla Cederbaum.
\pagebreak[3]

\section*{Structure of the paper}
In Section \ref{Assumptions_and_notation}, we fix notations and define when a Riemannian manifolds is said to be asymptotically Schwarzschildean. We prove in Section \ref{Additional_regularity} that the leaves of the CMC foliation are not completely off-center, i.\,e.\ there is an estimate on how far away the spheres can be from the coordinate origin (or coordinate CMC center of mass, if it is defined). The three versions of the first main result (Theorem \ref{evolution_of_the_leaves}, Corollary \ref{evolution_of_the_center}, and Theorem \ref{inf_evolution_of_the_leaves}) are stated precisely and proven in Section \ref{The_evolution}. In Section \ref{Calculating_the_Euclidean_coordinate_center_of_a_CMC_leaf_and___}, we prove that existence of the coordinate ADM center is equivalent to existence of the coordinate CMC center and that they coincide, if one and thus both of them exist.

\section{Assumptions and notation}\label{Assumptions_and_notation}
In order to study temporal foliations of four-dimensional spacetimes by three-dimensional spacelike slices and foliations (near infinity) of those slices by two-dimensional spheres, we will have to deal with different manifolds (of different or the same dimension) and different metrics on these manifolds, simultaneously. All four-dimensional quantities like the Lorentzian spacetime $(\uniM,\unig*)$, its Ricci and scalar curvatures $\uniric$ and $\unisc$, and all other derived quantities will carry a hat. In contrast, all three-dimensional quantities like the spacelike slices $(\outM,\outg*)$, its Ricci, scalar, exterior and mean curvature $\outric$, $\outsc$, $\outzFund$ and ${\outH}$, its future-pointing unit normal $\outnu$ and all other derived quantities carry a bar, while all two-dimensional quantities like the CMC leaf $(\M,\g*)$, its Ricci, scalar, exterior and mean curvature $\ric$, $\sc$, $\zFund$ and $\H$, its outer unit normal $\nu$ and all other derived quantities carry neither. 
When different three-dimensional manifolds or metrics are involved, then the upper left index will always denote the (real or artificial, see Construction \ref{Artificial_Lorentzian_manifolds}) time-index of the \lq current\rq\ time-slice. The only exceptions are the upper left indices $\euclideane$, $\schws$, and $a$, which refer to the Euclidean, the Schwarz\-schild, and the artificial metric (see Construction \ref{Artificial_Lorentzian_manifolds}), respectively. If different two-dimensional manifolds or metrics are involved, then the lower left index will always denote the mean curvature index $\Hradius$ of the current leaf $\M<\Hradius>$, i.\,e.\ the leaf with mean curvature $\H<\Hradius>=\nicefrac{{-}2}\Hradius+\nicefrac{4m}{\Hradius^2}$. The two-dimensional manifolds and metrics (and therefore other metric quantities) thereby \lq inherit\rq\ the time-index of the corresponding three-dimensional manifold. We abuse notation and suppress these indices, whenever it is clear from the contest which metric we refer to.

It should be noted that we interpret the second fundamental form and the normal vector of a hypersurface, as well as  the \lq lapse function\rq\ and the \lq shift vector\rq\ of a hypersurfaces arising as a leaf of a given deformation or foliation as quantities on the hypersurfaces (and thus as \lq lower\rq\ dimensional). For example, if $\outM[t]$ is a hypersurface in $\uniM$, then $\outnu[t]$ denotes its normal (and \emph{not} $\unisymbol\nu$). \smallskip\pagebreak[1]

Furthermore, we use upper case latin indices $\ii$, $\ij$, $\ik$, and $\il$ for the two-dimensional range $\lbrace2,3\rbrace$ and lower case latin indices $\oi$, $\oj$, $\ok$, and $\ol$ for the three-dimensional range $\lbrace 1,2,3\rbrace$. The Einstein summation convention is used accordingly. \pagebreak[3]\medskip

As mentioned, we frequently use foliations and evolutions. These are infinitesimally characterized by their lapse functions and their shift vectors.
\begin{definition}[Lapse functions, shift vectors]
Let $\theta>0$, $\sigma_0\in\R$, $I\supseteq\interval{\Hradius_0-\delta\Hradius}{\Hradius_0+\delta\Hradius}$ an interval, and let $(\outM,\outg*)$ be a Riemannian manifold. A smooth map $\DoATensor\Phi:I\times\M\to\outM$ is called \emph{deformation} of the closed hypersurface $\M=\M<\Hradius_0>=\Phi(\Hradius_0,\M)\subseteq\outM$, if $\DoATensor\Phi<\Hradius>(\cdot):=\Phi(\Hradius,\cdot)$ is a diffeomorphism onto its image $\M<\Hradius>:=\DoATensor\Phi<\Hradius>(\M)$ and if $\DoATensor\Phi<\Hradius_0>\equiv\id_{\M}$. The decomposition of $\spartial[\Hradius]@\Phi$ into its normal and tangential parts can be written as
\begin{align*} \partial[\Hradius]@\Phi ={}& {\rnu<\Hradius>}\,{\nu<\Hradius>} + {\rbeta<\Hradius>}, \end{align*}
where $\nu<\Hradius>$ is the outer unit normal to $\M<\Hradius>$. The function $\rnu<\Hradius>:\M<\Hradius>\to\R$ is called the \emph{lapse function} and the vector field $\rbeta<\Hradius>\in\X(\M<\Hradius>)$ is called the \emph{shift} of $\Phi$. If $\Phi$ is a diffeomorphism (resp.\ diffeomorphism onto its image), then it is called a \emph{foliation} (resp.\ a \emph{local foliation}).

In the setting of a Lorentzian manifold $(\uniM,\unig*)$ and a non-compact, spacelike hypersurface $\outM\subseteq\uniM$ the notions of deformation, foliation, lapse $\ralpha$ and shift $\outbeta$ are defined correspondingly.
\end{definition}\pagebreak[3]

As there are different definitions of \lq asymptotically Schwarzschildean\rq, we now describe the asymptotic assumptions we make. To rigorously define these and to shorten the statements in the following, we distinguish between the Riemannian, the initial data, and the foliation case.
\begin{definition}[Schwarzschild quantities]\label{Schwarzschild}
On $\R^3\setminus\{0\}$, the metric $\schwarzg*$ and the lapse function $\ralpha[\schws]*$ of the standard Schwarz\-schild timeslice are defined by
\[
 \schwarzg := (1+\frac m{2\rad})^4\,\eukg, \qquad\qquad
 \ralpha[\schws] := \frac{1-\frac{2m}r}{1+\frac{2m}r},
\]%
where $m>0$, $\eukg$ denotes the Euclidean metric, and $\rad:\R^3\setminus\{0\}\to\interval0\infty:\outsymbol x\mapsto\vert\outsymbol x\vert$ denotes the Euclidean distance to the origin.
\end{definition}\pagebreak[3]

\begin{definition}[\texorpdfstring{$\Ck^k$}{C-k}-{asymptotically} Schwarzschildean Riemannian manifolds]\label{Ck_asymptotically_Schwarzschildean}
Let ${\outve}>0$. A triple $(\outM,\outg,\outsymbol x)$ is called \emph{$\Ck^k$-asymptotically Schwarz\-schildean of order $1+{\outve}$}, if $(\outM,\outg)$ is a smooth manifold and $\outsymbol x:\outM\setminus\overline L\to\R^3$ is a chart of $\outM$ outside a compact set $\overline L$ such that there exists a constant $\oc\ge0$ with
\[
 \vert\frac{\partial*^{|\gamma|}(\outg_{ij}-\schwarzg_{ij})}{\partial*\outsymbol x^\gamma}\vert \le \frac{\oc}{\rad^{1+\vert\gamma\vert+{\outve}}}, \qquad\qquad\forall\,\vert\gamma\vert\le k,
\]%
where $\rad:\outM\setminus\overline L\to\interval0\infty:\outsymbol p\mapsto\vert\outsymbol x(\outsymbol p)\vert$ is the Euclidean coordinate distance to the coordinate origin.
\end{definition}\pagebreak[3]

\begin{definition}[\texorpdfstring{$\Ck^k$}{C-k}-{asymptotically} Schwarzschildean initial data sets]\label{Ck_asymptotically_Schwarzschildean_initial_data_set}
Let ${\outve}>0$ and let $(\outM,\outg*,\outsymbol x,\outzFund,\enden,\outsymbol J)$ be an initial data set, which means that $(\outM,\outg*)$ is a Riemannian manifold, $\outzFund$ a symmetric $(0,2)$-tensor, $\enden$ a function, and $\outsymbol J$ a one-form on $\outM$, respectively, satisfying the \emph{Einstein constraint equations}\footnote{We dropped the physical factor $8\pi$ for notational convenience.}
\[
 \outsc - \outtrtr\outzFund\outzFund + \outmc^2 = 2\enden, \qquad\qquad
 \outdiv(\outmc\;\outg*-\outzFund) = \outsymbol J,
\]
where $\outmc:=\outtr\,\outzFund$. $(\outM,\outg*,\outsymbol x,\outzFund,\enden,\outsymbol J,\ralpha)$ is called \emph{$\Ck^k$-asymptotically Schwarz\-schildean of order $1+{\outve}$} if $(\outM,\outg*,\outsymbol x)$ is $\Ck^k$-asymptotically Schwarz\-schildean of order $1+{\outve}$ and if $\ralpha$ is a function on $\outM$ such that there are constants $\outdelta\in\interval0{1+{\outve}}$ and $\oc\ge0$ with
\begin{equation*}
 \vert\frac{\partial*^{|\gamma|}\,(\ralpha-\ralpha[\schws])}{\partial*\outsymbol x^\gamma}\vert \le \frac{\oc}{\rad^{1+\vert\gamma\vert+{\outve}-\outdelta}}, \quad
 \vert\frac{\partial*^{|\gamma|}\;\outzFund_{ij}}{\partial*\outsymbol x^\gamma}\vert \le \frac{\oc}{\rad^{1+\vert\gamma\vert+\outdelta}} \quad
  \quad\forall\,\vert\gamma\vert<k \labeleq{decay_outzFund}
\end{equation*}
holds (in the chart $\outsymbol x$).
\end{definition}\pagebreak[1]
We remark that the asymptotic decay in \eqref{decay_outzFund} is a generalization of the one frequently used in the literature, where one assumes $\outdelta=1$, i.\,e.\ $\vert\outzFund\vert\le\nicefrac C{\rad^2}$ and $\vert\ralpha-\ralpha[\schws]\vert\le\nicefrac C{\rad^{\outve}}$. In particular, this is the minimal assumption for which one can hope for existence of the ADM linear momentum -- without imposing any asymptotic symmetry assumption as for example the Regge-Teitelboim conditions.\pagebreak[3]

\begin{definition}[\texorpdfstring{$\Ck^k$}{C-k}-{asymptotically} Schwarzschildean temporal foliations]\label{Ck_asymptotically_Schwarzschildean_foliation}
Let ${\outve}>0$. The smooth level sets $\outM[t]:=\time^{-1}(t)$ of a smooth function $\time$ on a four-dimensional Lorentz manifold $(\uniM,\unig*)$ are called a \emph{$\Ck^k$-asymptotically Schwarz\-schildean temporal fo\-li\-a\-tion of order $1+\ve$}, if the gradient of $\time$ is everywhere time-like, and if there is a chart $(\time,\unisymbol x):\unisymbol U\subseteq\uniM\to\R\times\R^3$ of $\uniM$, such that $(\outM[t],\outg[t]*,{\DoAoutTensor x[t]}:=\unisymbol x|_{\outM[t]},\outzFund[t]*,\enden[t],\momenden[t],\ralpha[t])$ is a $\Ck^k$-asymptotically Schwarzschildean initial data set of order $1+{\outve}$ for all $t$ with not necessarily uniform $\oc[t]$ and $\outdelta[t]$. Here, the corresponding second fundamental form ${\outzFund[t]*}$, the lapse function ${\ralpha[t]}$, the energy-density ${\enden[t]}$ and the momentum density ${\momenden[t]}$ are defined by
\begin{align*}
 \ralpha[t] :={}& \sqrt{{-}\unig(\partial[\time],\partial[\time])}, &
 \outzFund[t]_{ij} :={}& \frac{{-}1}{2\ralpha[t]}\partial[t]@{\outg[t]_{ij}}, &
 \left.\partial[\time]\right|_{\outM[t]} ={}& \ralpha[t]\;\outnu[t], \\
 \enden[t] :={}& \uniric(\outnu[t],\outnu[t])+\frac\unisc2, &
 \momenden[t]* :={}& \uniric(\outnu[t],\cdot),
\end{align*}
respectively, where $\outnu[t]$ is the future-pointing unit normal to $\outM[t]$. If the constants $\oc[t]$ and $\outdelta[t]$ of the decay assumptions can be chosen independently of $t$, then the temporal fo\-li\-a\-tion is called \emph{uniformly $\Ck^k$-asymptotically Schwarz\-schildean}.
\end{definition}\pagebreak[3]

It is important to note that the decay assumptions made in Definition \ref{Ck_asymptotically_Schwarzschildean_initial_data_set} do \emph{not} imply that the ADM linear momentum $\impuls$ is well-defined. In particular, there is no hope to prove $\spartial[t]@\centerz=\nicefrac\impuls m$ under these assumptions. This is why we use the pseudo quasi-local linear momentum $\impuls<\Hradius>$ for which we prove $\spartial[t]@\centerz=\nicefrac{\impuls<\Hradius>}m$. This is explained in more detail in Remark \ref{definition_of_linear_momentum}. Beside the obvious advantage of achieving a more general result, the approach with these weak decay assumptions allows us to prove equality of the ADM and CMC center of mass under other assumptions than those in the literature and to use these results in the setting of the examples in \cite{cederbaumnerz2013_examples}. \pagebreak[3]

\begin{remark}[The asymptotically flat setting]\label{Asymptotically_flat_setting_2}
As explained in Remark \ref{Asymptotically_flat_setting}, the results of Section \ref{The_evolution} carry over to the more general setting of $\mathcal C^2_{\frac12+\outve}$-asymptotically flat $(\outM,\outg*,\outsymbol x,\outzFund,\enden,\outsymbol J,\ralpha)$, i.\,e.\ a initial data set $(\outM,\outg*,\outzFund,\enden,\outsymbol J)$ with a chart $\outsymbol x:\outM\setminus\outsymbol L\to\R^3$ such that
\[
 \vert\frac{\partial*^{|\gamma|}(\outg_{ij}-\eukg_{ij})}{\partial*\outsymbol x^\gamma}\vert \le \frac{\oc}{\rad^{\frac12+\vert\gamma\vert+{\outve}}}, \qquad\qquad\forall\,\vert\gamma\vert\le 2,
\]%
and where the second fundamental form $\outzFund$ and the \emph{lapse function} $\ralpha$ satisfy
\[
 \vert\frac{\partial*^{|\gamma|}\,(\ralpha-\ralpha[\schws])}{\partial*\outsymbol x^\gamma}\vert \le \frac{\oc}{\rad^{\frac12+\vert\gamma\vert+\outve}}, \quad
 \vert\frac{\partial*^{|\gamma|}\;\outzFund_{ij}}{\partial*\outsymbol x^\gamma}\vert \le \frac{\oc}{\rad^{\frac32+\vert\gamma\vert+\outve}} \quad
  \quad\forall\,\vert\gamma\vert<2 
\]
\end{remark}

\section{CMC leaves are almost asymptotically concentric}\label{Additional_regularity}
In this section, we assume that $(\outM,\outg*,\outsymbol x)$ is an three-dimensional, asymptotically Schwarz\-schildean Riemannian manifold with mass $m>0$. Due to \cite{metzger2007foliations}, there exists a foliation $\Phi:\interval{\Hradius_0}\infty\times\sphere^2\to\outM$ of $\outM$ near infinity by closed surfaces with constant mean curvature $\H[\Hradius]\equiv\nicefrac{{-}2}\Hradius+\nicefrac{4m}{\Hradius^2}$. We show in this section, that the leaves of the CMC foliation for one fixed time-slice $\outM:=\outM[t]$ are almost concentric, i.\,e.\ the Euclidean coordinate center $\centerz<\Hradius>$ of each leaf $\M<\Hradius>$ (as defined in \eqref{Euclidean_Center_of_mass}) is of order $\Hradius^{1-{\outve}}$. Cederbaum-Nerz \cite{cederbaumnerz2013_examples} constructed examples proving that these results are sharp. It should be mentioned that an additional assumption on the scalar curvature $\outsc$ implies better estimates. More precisely by \cite[Lemma F.1]{metzger_eichmair_2012_unique}, the center of mass is well-defined if $\outsc$ fulfills the Regge-Teitelboim conditions (is asymptotically antisymmetric) and Corollary \ref{euclidean_coordinate_center_of_a_leave} then ensures that the Euclidean coordinate centers are bounded.\smallskip

To prove that a given asymptotically Schwarzschildean three-dimensional Riemannian manifold $(\outM,\outg*)$ can be foliated by stable CMC spheres, Metzger defined artificial metrics ${\aoutg[\tau]}:=\schwarzg*+\tau(\outg*-\schwarzg*)$ and proved that the set 
\[ I := \lbrace \tau\in\interval*0*1 \;\middle|\; \text{a CMC foliation near infinity exists with respect to } \aoutg[\tau]* \rbrace \]
is non-empty, open and closed in (and therefor equal to) $\interval*0*1$ \cite{metzger2007foliations}. In doing so, he has to prove several decays conditions of these spheres. All in all, we cite the following special case:
\begin{theorem}[Special case of {\cite[Thm.~1.1]{metzger2007foliations}}]\label{existence}
There exists a constant $\Hradius_0=\Cof{\Hradius_0}[m][{\outve}][\oc]{<\infty}$ and a $\Ck^1$-map $\aoutsymbol\Phi:\interval*0*1\times\interval{\Hradius_0}\infty\times\mathcal S^2\to\outM$ such that the \emph{geometric spheres} $\M<\Hradius>[\atime\,]:=\aoutsymbol\Phi(\atime,\Hradius,\mathcal S^2)$ are stable, have constant mean curvature $\H<\Hradius>[\atime]\equiv\nicefrac{{-}2}\Hradius+\nicefrac{4m}{\Hradius^2}$ with respect to the artificial metric $\aoutg[\atime]*:=\schwarzg+\atime\,(\outg*-\schwarzg)$. Furthermore, there is a constant $C=\Cof[m][{\outve}][\oc]$ not depending on $\atime$ or $\Hradius$ such that
\begin{equation*}
 \vert\volume{\M<\Hradius>[\atime]}-4\pi\Hradius^2\vert \le C\Hradius, \qquad
 \Vert\zFundtrf<\Hradius>[\atime]_{}\Vert_{\Lp^\infty(\M<\Hradius>[\atime])} + \Vert\levi<\Hradius>[\atime]*\zFundtrf<\Hradius>[\atime]_{}\Vert_{\Lp^2(\M<\Hradius>[\atime])} \le \frac C{\Hradius^2}, \labeleq{estimates_for_k}
\end{equation*}
where $\zFundtrf<\Hradius>[\atime]_{}:=\zFund<\Hradius>[\atime]_{}-\nicefrac12\;\H<\Hradius>[\atime]\;\g<\Hradius>[\atime]*$ denotes the trace-free part of the exterior curvature ${\zFund<\Hradius>[\atime]}$, ${\g<\Hradius>[\atime]*}$ denotes the induced metric -- both with respect to $\aoutg[\atime]*$ -- and $\levi<\Hradius>[\atime]*$ is the Levi-Civita connection. Additionally, there is a constant $C_S=\Cof{\c_S}[m][{\outve}][\oc]$ not depending on $\atime$ or $\Hradius$ such that
\begin{equation*}
 \Vert f\Vert_{\Lp^2(\M<\Hradius>[\atime])}
	\le C_S(\Hradius\Vert\vert\levi<\Hradius>[\atime]*f\vert\Vert_{\Lp^1(\M<\Hradius>[\atime])} + \Vert f\Vert_{\Lp^1(\M<\Hradius>[\atime])}) \qquad\qquad\forall\,f\in\Ck^1(\M<\Hradius>[\atime]).
 \labeleq{Erst e_Sobolev}
\end{equation*}
Furthermore, $\DoAaoutTensor\Phi[\atime]:=\aoutsymbol\Phi(\atime,\cdot,\cdot)$ is a $\Ck^1$-foliation of $\outM$ near infinity.
\end{theorem}
By using DeLellis-M{\"u}ller \cite{DeLellisMueller_OptimalRigidityEstimates}, Metzger concluded that for any $\atime\in\interval*0*1$ and $\Hradius>\Hradius_0$ there is a conformal parametrization $\DoATensor\psi<\Hradius>[\atime]:\mathcal S^2\to\M<\Hradius>[\atime]$ with
\begin{align*}
 \Vert \DoATensor\psi<\Hradius>[\atime] - (\centerz<\Hradius>[\atime]+\Hradius\,\id) \Vert_{\Hk^2(\mathcal S^2)} \le{}& C\Hradius, &
 \Vert h^2 - 1\Vert_{\Hk^1(\mathcal S^2)} \le{}& C, \labeleq{estimates_on_conformal_parametrization}\\
 \Vert N - \normal<\Hradius>[\atime]\circ\DoATensor\psi<\Hradius>[\atime]\Vert_{\Hk^1(\mathcal S^2)} \le{}& C, \labeleq{estimates_on_the_normal}
\end{align*}
where $\id:=\outsymbol x^{-1}\circ\mathcal i$, $\mathcal i$ is the standard embedding of the Euclidean sphere $\mathcal S^2_\sigma(0)$ to $\R^3$, $\centerz<\Hradius>[\atime]\in\R^3$ is the Euclidean coordinate center of $\M<\Hradius>[\atime]$, $\DoATensor<\Hradius>[\atime]h$ is the conformal factor $\DoATensor\psi<\Hradius>[\atime]^*\g[\euce]*=\DoATensor<\Hradius>[\atime] h^2\;\g<\Hradius>[\atime]*$ for the standard metric $\g[\euce]*$ on $\mathcal S^2$, $N$ is the Euclidean outer unit normal to $\mathcal S^2\subseteq\R^3$, $\normal<\Hradius>[\atime]$ is the normal to $\M<\Hradius>[\atime]$ with respect to $\aoutg[\atime]*$ and the Sobolev norms are defined in Definition \ref{sobolev_norms}. Note that, in contrast to \cite{metzger2007foliations}, we use scale-invariant Sobolev norms.
\begin{definition}[Sobolev norms]\label{sobolev_norms}
For any tensor $T$, $k\ge1$ and $p\in\interval*1*\infty$ the Sobolev norm on the surface $\M<\Hradius>[\atime]$ is iteratively defined by
\[
 \Vert T\Vert_{\Wkp^{k,p}(\M<\Hradius>[\atime])} := \Vert\vert T\vert_{\g<\Hradius>[\atime]*}\Vert_{\Lp^p(\M<\Hradius>[\atime])} + \Hradius\Vert\levi*T\Vert_{\Wkp^{k-1,p}(\M<\Hradius>[\atime])},
\]
where $\Wkp^{0,p}(\M<\Hradius>[\atime]):=\Lp^p(\M<\Hradius>[\atime])$. Additionally, $\Hk^k(\M<\Hradius>[\atime]):=\Wkp^{k,2}(\M<\Hradius>[\atime])$ in the case $p=2$.
\end{definition}
We mention that Metzger also proved that the Euclidean coordinate centers of the leaves are of order $\Hradius$, with controlled constant $<1$. As this is not enough for our context, we prove the following inequality.
\begin{proposition}[Leaves are almost concentric]\label{leaves_are_almost_concentric}
Let $(\outM,\outg*,\outsymbol x)$ be $\Ck^2$-asymptotically Schwarzschildean of order $1+{\outve}$ with ${\outve}>0$. There are constants $C=\Cof[m][{\outve}][\oc]$ and $\Hradius_0=\Cof{\Hradius_0}[m][{\outve}][\oc]$ such that the Euclidean coordinate center $\centerz<\Hradius>$ of each $\M<\Hradius>$ with $\Hradius\ge\Hradius_0$ fulfills
\[ \vert\centerz<\Hradius>\vert \le C\Hradius^{1-{\outve}}, \]
where $\lbrace\M<\Hradius>\rbrace_\Hradius$ denotes the foliation of $\outM$ near infinity by stable spheres $\M<\Hradius>$ with constant mean curvature ${\H[\Hradius]}\equiv\nicefrac{{-}2}\Hradius+\nicefrac{4m}{\Hradius^2}$.
\end{proposition}\pagebreak[3]
In analogy to Metzger's proof of \cite[Thm.~1.1]{metzger2007foliations}, we prove Proposition \ref{leaves_are_almost_concentric} by showing that the set of \lq almost concentric spheres\rq\ is non-empty, open and closed in (and therefor equal to) $\interval*0*1\times\interval{\Hradius_0}\infty$ for some $\Hradius_0<\infty$. Thus, we define for $c\ge0$
\[ I_{\c} := \lbrace (\atime,\Hradius)\in\interval*0*1\times\interval{\Hradius_0}\infty \;\middle|\; {\centerz<\Hradius>[\atime]}<\c\,\Hradius^{1-{\outve}}\rbrace. \]
Therefore, $(\tau,\sigma)\in I_{\c}$ corresponds to a quantitative estimate of being almost concentric.
As the CMC foliation for the Schwarzschild metric ($\atime=0$) is given by the concentric spheres $\mathcal S^2_{\rradius(\Hradius)}(0)$, we see that $I_{\c}\supseteq\lbrace0\rbrace\times\interval{\Hradius_0}\infty$ for any $\c\ge0$. By continuity of $\aoutsymbol\Phi$ (see Theorem \ref{existence}), we conclude that $I_{\c}$ is closed.

To show that $I_{\c}$ is open in $\interval*0*1\times\interval{\Hradius_0}\infty$ for some ${\c}<\infty$, we prove estimates on the evolution of the CMC spheres $\M<\Hradius>[\atime]$ in \lq$\tau$-direction\rq. To do so, we have to show inequalities for the corresponding lapse function. As we will show, the lapse function is defined by its image under the stability operator. The stability operator of a surface $\M^n\hookrightarrow\R^{n+1}$ can be defined (or interpreted) as the linearization at $0$ of the graph mean curvature map $\Hk^2(\M)\mapsto\Lp^2(\M):f\mapsto\H(\graph f)$, where $\H(\graph f)$ denotes the mean curvature of the $\graph f:=\lbrace p+f(p)\;\nu\,|\,p\in\M\rbrace$. This map is well-defined if $\Vert f\Vert_{\Hk^2(\M)}$ is small and $\M$ regular. It is well-known that the stability operator can be written in the following form.
\begin{definition}[The stability operator]
The stability operator $\jacobiext<\Hradius>[\atime]*:\Hk^2(\M<\Hradius>[\atime])\to\Lp^2(\M<\Hradius>[\atime])$ on the surface $\M<\Hradius>[\atime]$ is given by -- omitting the $\atime$ and $\Hradius$ indices --
\[ \jacobiext*f := \laplace f + (\vert\zFund\vert_{\g*}^2+\aoutric(\normal,\normal))f, \]
where $\aoutric$ is the Ricci curvature of $\outM$ and $\normal$ is the normal to $\M$ -- both with respect to the metric $\aoutg[\atime]*$. 
\end{definition}
To conclude estimates for the stability operator, we have to control its eigenvalues.
\begin{lemma}[Smallest eigenvalues of the stability operator]\label{the_lowest_eigenvalues_of_the_stability_operator}
There are constants $C=\Cof[m][{\outve}][\oc][\c]$ and $\Cof{\Hradius_0}[m][{\outve}][\oc][\c]$ such that every eigenvalue $\lambda$ of the stability operator $\jacobiext<\Hradius>[\atime]*$ on $\M<\Hradius>[\atime]$ with $(\atime,\Hradius)\in I_{\c}$ and $\Hradius>\Hradius_0$ fulfills
\[ \vert\lambda\vert \le \frac1{\Hradius^2} \iff \vert\lambda-\frac{6m}{\Hradius^3}\vert \le \frac C{\Hradius^{3+{\outve}}}. \]
\end{lemma}
\begin{proof}
Let $(\atime,\Hradius)\in I_{\c}$ and let us omit the corresponding indices in the following. Furthermore, let $\normal$ be the normal to $\M$, $\aoutric*$ the Ricci curvature of $\outM$ -- both with respect to the metric $\aoutg[\atime]*$ --, $N(\outsymbol x):=\nicefrac{(\outsymbol x-\centerz)}{\vert \outsymbol x-\centerz\vert}$ the radial direction with respect to the Euclidean coordinate center of $\M$ and $\radialdirection(\outsymbol x):=\nicefrac{\outsymbol x}{\vert \outsymbol x\vert}$ the radial direction with respect to the coordinate origin. We obtain
\begin{align*}
 \vert\aoutric(\nu,\nu) + \frac{2m}{\Hradius^3}\vert
	\le{}& \vert\aoutric(\radialdirection,\radialdirection) + \frac{2m}{\Hradius^3}\vert + \frac C{\Hradius^3}\,\vert N-\normal\vert + \frac C{\Hradius^3}\,\vert N-\radialdirection\vert \\
	\le{}& \frac C{\Hradius^{3+{\outve}}} + \frac C{\Hradius^3}\,\vert N-\normal\vert + \frac C{\Hradius^4}\,\vert\centerz\vert.
\end{align*}
By assumption on $(\atime,\Hradius)\in I_{\c}$, the estimate \eqref{estimates_on_the_normal} on the normal and the Sobolev inequality therefore imply that
\[ \Vert\aoutric(\nu,\nu) + \frac{2m}{\Hradius^3}\Vert_{\Lp^4(\M)} \le \frac C{\Hradius^{\frac52+{\outve}}}\labeleq{Ricci_inequality}. \]
By the estimates \eqref{estimates_on_conformal_parametrization} of the approximation $\M[\Hradius]\approx\mathcal S^2_{\Hradius}(\centerz)$ and the assumptions on $\outg$, we conclude that the first three eigenvalues $\lambda_1$, $\lambda_2$ and $\lambda_3$ of the Laplace fulfill 
\begin{align*}
 \vert\lambda_i-\frac2{\Hradius^2}\,(1-\frac{2m}\Hradius)\vert\le{}&\frac C{\Hradius^{3+{\outve}}}.
\end{align*}
The desired inequality now follows by the estimates \eqref{estimates_for_k} on $\zFundtrf$.
\end{proof}\pagebreak[2]

Metzger now uses the implicit function theorem to ensure the existence of the CMC spheres. Thereby, he obtains the regularity of spheres close to a regular one by proving that any sphere with constant mean curvature is regular. We obtain the estimate on the Euclidean coordinate center by taking into account the lapse functions and their regularity. Thus, we now decompose the derivative $\spartial[\sigma]@{\outsymbol\Phi}$ (and later $\spartial[\atime]@{\outsymbol\Phi}$) of $\outsymbol\Phi$ of Theorem \ref{existence} into lapse and shift:
\begin{defiundlemma}[Regularity of the lapse functions in space]\label{regularity_of_the_radial_lapse_function}
Let $\rnu:=\rnu<\Hradius>[\atime]$ be the lapse function in $\Hradius$-direc\-tion (lapse function in space) of $\aoutsymbol\Phi$ on $\M:=\M<\Hradius>[\atime]$, i.\,e.
\[ \left.\partial[\Hradius]@{\outsymbol\Phi}\right|_{\M} =: \rnu\,\normal + \rbeta, \]
with $\rbeta:=\rbeta<\Hradius>[\atime]\in\X(\M)$ and where $\normal:=\normal<\Hradius>[\atime]$ denotes the normal to $\M$ with respect to the metric ${\aoutg[\atime]*}$. There are constants $C=\Cof[m][{\outve}][\oc][\c]$ and $\Cof{\Hradius_0}[m][{\outve}][\oc][\c]$ such that for any $(\atime,\Hradius)\in I_{\c}$ with $\Hradius>\Hradius_0$
\[ \Vert \rnu - 1 \Vert_{\Hk^2(\M)} \le C\Hradius^{1-{\outve}}, \qquad\qquad
 \Vert \rnu - 1 \Vert_{\Wkp^{1,\infty}(\M)} \le \frac C{\Hradius^{\outve}}. \]
\end{defiundlemma}
\begin{proof}
By the definition of the lapse function $\rnu$ and the stability operator $\jacobiext*$, we conclude from inequality \eqref{Ricci_inequality} that
\begin{align*}\hspace{4em}&\hspace{-4em}
 \Vert\jacobiext*(\rnu-(1+\frac m\Hradius))\Vert_{\Lp^4(\M)} \\
  \le{}& \Vert\partial[\Hradius]({-}\frac2\Hradius+\frac{4m}{\Hradius^2}) - (1+\frac m\Hradius)\;(\frac2{\Hradius^2}-\frac{10m}{\Hradius^3})\Vert_{\Lp^4(\M)} + \frac C{\Hradius^{\frac52+{\outve}}} \\
	\le{}& \frac C{\Hradius^{\frac52+{\outve}}}.
\end{align*}
By Lemma \ref{the_lowest_eigenvalues_of_the_stability_operator}, we conclude the desired inequalities for $\rnu$ from the regularity of the Laplace operator.
\end{proof}\pagebreak[2]
Now, we have to prove an analogous lemma for the lapse function $\arnu$ on $\M=\M<\Hradius>[\atime]$ in $\atime$-direction (\emph{lapse function in spacetime}), i.\,e.
\begin{align*}
 \left.\partial[\atime]@{\aoutsymbol\Phi}\right|_{\M} =:{}& \arnu\,\normal + \atbeta,
\end{align*}
where $\atbeta:=\atbeta<\Hradius>[\atime]\in\X(\M)$ is the corresponding shift vector and $\normal:=\normal<\Hradius>[\atime]$ denotes the normal to $\M$ with respect to the metric $\aoutg[\atime]*$. Therefore, we have to estimate $\jacobiext<\Hradius>[\atime](\arnu<\Hradius>[\atime])$, where $\jacobiext<\Hradius>[\atime]$ is the stability operator of the surface $\M<\Hradius>[\atime]$ (with respect to the metric ${\aoutg[\atime]}$). As first step, we see that
\[ \partial[\atime]@{(\H<\Hradius>[\atime])} \equiv \partial[\atime](\frac{{-}2}\Hradius+\frac{4m}{\Hradius^2}) \equiv 0. \]
By definition, $\H<\Hradius>[\atime]$ is the mean curvature of the surface $\M<\Hradius>[\atime]$ with respect to the metric ${\aoutg[\atime]}$. Thus, by changing $\atime$ to $\atime+\delta\atime$, we do not only change the surface, but also the corresponding metric from ${\aoutg[\atime]}$ to ${\aoutg[\atime+\delta\atime]}$. To distinguish between these \lq two changes\rq, we define $\H<\Hradius>[\atime](\varrho)$ to be the mean curvature of the surface $\M<\Hradius>[\atime]$ with respect to the metric ${\aoutg[\varrho]}$. By the rules for the Lie derivative, we conclude
\begin{equation*}
 0 \equiv \partial[\atime]@{(\H<\Hradius>[\atime])}
	\equiv \partial[\atime]@{(\H<\Hradius>[\atime](\varrho))}<\varrho=\atime> + \partial[\varrho]@{(\H<\Hradius>[\atime](\varrho))}<\varrho=\atime>
	\equiv \jacobiext<\Hradius>[\atime](\arnu<\Hradius>[\atime]) + \partial[\varrho]@{(\H<\Hradius>[\atime](\varrho))}<\varrho=\atime>. \labeleq{stability_operator_on_arnu}
\end{equation*}
Therefore, we have to estimate the last derivative in order to prove the regularity of the lapse function in spacetime. To do so, we want to use the following proposition to characterize this derivative. The proposition is formulated in a more abstract way, because we also need it in Section \ref{The_evolution}.\pagebreak[2]
\begin{proposition}[Time derivative of the mean curvature]\label{change_of_meancurvature}
Let $\Phi:I\times\outM\to\uniM$ be an orthogonal foliation of an arbitrary Lorentzian manifold by spacelike hypersurfaces $\outM[t]:=\Phi(t,\outM)$, i.\,e.
\begin{align*} \partial[t]@\Phi ={}& \ralpha[t]\;\outnu[t] \end{align*}
for the lapse function $\ralpha[t]$ on $\outM[t]$ and the future-pointing unit normal $\outnu[t]$ to $\outM[t]$. Let further $\M\subseteq\outM$ be a closed hypersurface in $\outM$. If we denote by $\H[t]$ the mean curvature of the hypersurface $\M[t]:=\Phi(t,\M)$ in $\outM[t]$, then -- omitting the index $t$ --
\[ \partial@{\hspace{.05em}\H[t]}[t]
  = \ralpha\,(\momenden(\normal) - \div\outzFund_{\nu} + \trtr{\zFund}{\outzFund}) + (D_{\normal}\ralpha)\,\tr\outzFund - 2\outzFund_{\nu}(\levi*\ralpha), \]
where $\nu$ denotes the outer normal of $\M=\M[t]$ in $\outM=\outM[t]$, $\zFund$ the exterior curvature of $\M$ in $\outM$, $\outzFund$ the exterior curvature of $\outM$ in $\uniM$, $\outzFund_{\nu}(\cdot):={\outzFund}(\normal,\cdot)$ is a one-form on $\M$, $\outsymbol J(\normal):=\uniric(\outnu,\normal)$ and the $(0,2)$-tensor $\trzd\zFund\outzFund$ on $\M$ is defined by $(\trzd\zFund\outzFund)_{\ii\!\ij}:=\zFund_{\ii\!\ik}\g^{\ik\!\il}\outzFund_{\il\hspace{-.05em}\ij}$.
\end{proposition}
\begin{proof}
We fix $t_0$, omit all unnecessary indices, and identify $\M$ with $\Phi(t,\M)$. We note that the desired equality is a local identity, in particular, it is sufficient to prove it in a chart. Thus, we construct an adapted chart $\unisymbol x$ of $\uniM$ in a neighborhood of an arbitrary point $p\in\M$. Choose in a neighborhood $U\subseteq\M$ of $p$ a chart $x$, such that the metric $\g$ induced on $\M$ by $\Phi$ and $\outg[t]$ is normal in this point $p\in U$, i.\,e.
\begin{align*} \g_{\!\ii\!\ij}(p) ={}& \eukg_{\!\ii\!\ij}, & \partial_\ii@{\g_{\!\ii\!\ik}}(p) ={}& 0. \end{align*}
Now, choose a chart $\outsymbol x$ on a neighborhood of $p$ in $\outM$. This can be done in such a way, that $\outsymbol x|_{\M}=(0,x)$ and such that $\spartial[\outsymbol x]_1$ is everywhere orthogonal to $\spartial[\outsymbol x]_\ii$ for $\ii>1$. Finally, choose a chart $\unisymbol x$ on a neighborhood of $p$ in $\uniM$, such that $\unisymbol x_0|_{\outM[t]}\equiv t$ for any $t$ and $\unisymbol x|_{\outM[t_0]}\equiv(t_0,\outsymbol x)$ and that $\hat e_0:=\spartial[\hat x]_0$ is everywhere orthogonal to $\spartial[\hat x]_\oi$ for $\oi>0$. Let $f$ be the orthogonal projection of $e_1:=\spartial[\outsymbol x]_1$ on the unit normal $\nu[t]$ of $\M[t]$ in $\outM[t]$, i.\,e.\ $f:=\unisymbol e_1-\unig^{\ii\!\ij}\unig_{1\ii}\unisymbol e_{\ij}$. In particular, we get $\tilde\rnu=\rnu$ and $\spartial[t]@{\tilde\rnu}=\spartial[t]@\rnu$ on $\M$, where $\tilde\rnu:=\sqrt{\unig(f,f)}$ and $\rnu:=\sqrt{\unig(e_1,e_1)}$.\smallskip\pagebreak[1]

Denoting the Lie derivative of a tensor $T$ in direction of a vector field $X$ by $\lieD XT$, we see
\begin{align*}
 \tr(\lieD{\hat e_0}\zFund)
 ={}& \frac1{2\rnu^2}\partial_0@\rnu\,\tr(\lieD{\hat e_1}\unig*) - \frac1{2\rnu}\tr(\lieD{\hat e_1}\lieD{\hat e_0}\unig*) - \frac1{2\rnu}\tr(\lieD{[\hat e_0,f]}\unig*) \\
 ={}& \frac1\rnu(\ralpha\,\tr(\lieD{\outsymbol e_1}\outzFund) - \partial_0@\rnu\,\H + \partial_1@\ralpha\,\troutzFund
			- \g^{rs} (\vphantom{\frac\cdot\cdot} D_r(\ralpha\rnu\outzFund_{\nu s})+ D_s(\ralpha\rnu\outzFund_{\nu r}))) \\
 ={}& \ralpha\outzFund_{\nu\nu}\,\H + \frac\ralpha\rnu(\tr(\lieD{\outsymbol e_1}\outzFund)-2\outzFundnu(\levi*\rnu))
			+ \partial_1@\ralpha\frac\troutzFund\rnu - 2\outzFundnu(\levi*\ralpha) -2\ralpha\,\div\outzFundnu\!.
\end{align*}
This implies
\begin{align*}
 \lieD{\hat e_0}{\H}
	={}& \tr(\lieD{\hat e_0}\zFund) - \g^{pr}\,\partial_0@{\g_{rs}}\,\g^{sq}\,\zFund_{pq}
	= \tr(\lieD{\hat e_0}\zFund) + 2\ralpha\,\g^{pr}\,\outzFund_{rs}\,\g^{sq}\,\zFund_{pq} \\
	={}& \tr(\lieD{\hat e_0}\zFund) - \frac\ralpha\rnu\,\g^{pr}\,\outzFund_{rs}\,\g^{sq}\,\partial_0@{\g_{pq}} \\
	={}& \ralpha\,\outzFund_{\nu\nu}\H + \frac1\rnu(\ralpha\,\lieD{e_1}{(\troutzFund)} - 2\ralpha\,\outzFundnu(\levi*\rnu) + \partial_1@\ralpha\,\troutzFund) \\
			& - 2\outzFundnu(\levi*\ralpha)-2\ralpha\,\div\outzFundnu\!.
\end{align*}
As a first step, we thus conclude
\begin{equation*}
 \lieD{\hat e_0}{(\rnu\,\H)}= \lieD{e_1}{(\ralpha\,\troutzFund)} - 2\rnu\,\outzFundnu(\levi*\ralpha)-2\ralpha\,\outzFundnu(\levi*\rnu)-2\ralpha\,\rnu\,\div\outzFundnu\!. \labeleq{partial_0_H}
\end{equation*}

By definition of the exterior curvature, we know
\[ \outzFund(\normal,\normal) = \frac{{-}1}{\rnu\ralpha}\,\partial_0@\rnu, \]
which by \eqref{partial_0_H} leads to
\begin{equation*}
 \rnu\,\partial_0@\H = \partial_1@{(\ralpha\,\troutzFund)} - 2\rnu\,\outzFundnu(\levi*\ralpha)-2\ralpha\,\outzFundnu(\levi*\rnu)-2\ralpha\,\rnu\,\div\outzFundnu - \H\,\partial_0@\rnu. \labeleq{partial_0_H__2}
\end{equation*}
By the Codazzi equation, we furthermore know
\begin{align*}
 \rnu\,\momenden(\normal) ={}& \partial_1@{\outmc} - (\outdiv\outzFund)(e_1)
  = \partial_1@{(\troutzFund)} - 2\outzFundnu(\levi*\rnu) + \rnu(\outzFund_{\nu\nu}\!\H- \div\outzFund_{\nu} - \trtr\zFund\outzFund).
\end{align*}
Combining this with \eqref{partial_0_H__2}, we get the desired identity.
\end{proof}\pagebreak[3]
To use Proposition \ref{change_of_meancurvature} in the current setting, we have to construct a suitable Lorentzian manifold. To do so, we recall that Metzger considers artificial metrics ${\aoutg[\tau]}:=\schwarzg*+\tau(\outg*-\schwarzg*)$ ($\atime\in\interval*0*1$). We extend this idea and construct an \emph{artificial spacetime}.
\begin{construction}[Artificial Lorentzian manifolds]\label{Artificial_Lorentzian_manifolds}
Let $(\outM,\outg*,\outsymbol x)$ be an asymptotically Schwarzschildean three-dimensional Riemannian manifold. The \emph{artificial Lorentzian manifold} $(\uniM,\unig*)$ of $(\outM,\outg*)$ is defined by
\[
 \uniM[a] := \interval*0*1\times\outM, \qquad\qquad
 \unig[a]* := {-}\d\DoAuniTensor\tau\otimes\d\DoAuniTensor\tau+\schwarzg+\DoAuniTensor\tau\,(\outg*-\schwarzg),
\]
where the Schwarzschild metric is defined with respect to the given chart $\outsymbol x$ and $\atime:\uniM[a]\to\interval*0*1:(\atime,\outsymbol p)\mapsto\atime$ denotes the \emph{artificial time}. In particular, the three-dimensional lapse function $\ralpha$ fulfills $\ralpha\equiv1$.
\end{construction}\pagebreak[3]
Now we can use Proposition \ref{change_of_meancurvature} in order to get estimates of $\jacobiext<\Hradius>[\atime]\,(\arnu<\Hradius>[\atime])$.
\begin{defiundlemma}[\texorpdfstring{$\atime$-}{\lq artificial time\rq\ }derivative of the mean curvature]\label{change_of_meancurvature_sc}
Let $\arnu:=\rnu<\Hradius>[\atime]$ be the lapse function in $\atime$-direction (lapse function in spacetime) of $\aoutsymbol\Phi$ on $\M:=\M<\Hradius>[\atime]$, i.\,e.
\[ \left.\partial[\atime]@{\aoutsymbol\Phi}\right|_{\M} =: \arnu\,\normal + \atbeta, \]
where $\atbeta:=\atbeta<\Hradius>[\atime]\in\X(\M)$ is the corresponding shift vector and $\normal:=\normal<\Hradius>[\atime]$ denotes the normal to $\M$ with respect to the metric $\aoutg[\atime]*$. If $\H<\Hradius>[\atime](\varrho)$ is the mean curvature of $\M<\Hradius>[\atime]$ with respect to the metric $\aoutg[\varrho]*$, we obtain
\[ \vert\jacobiext<\Hradius>[\atime_{\!0}]\,\arnu<\Hradius>[\atime_{\!0}]\vert = \vert\partial[\varrho]<\varrho=\atime_0>@{\H<\Hradius>[\atime_{\!0}](\varrho)}\vert \le \frac C{\Hradius^{2+{\outve}}} \qquad\forall\,(\atime_0,\Hradius)\in I_{\c},\;\Hradius>\Hradius_0 \]
where $C=\Cof[m][{\outve}][\oc][\c]$ and $\Hradius_0=\Cof{\Hradius_0}[m][{\outve}][\oc][\c]$ are constants not depending on $\atime$.
\end{defiundlemma}
\begin{proof}
Using Proposition \ref{change_of_meancurvature} for the artificial Lorentzian manifold $(\uniM,\unig)$ constructed above and $\M:=\M<\Hradius>[\atime]$, we conclude -- omitting the indices $\atime$ and $\Hradius$ -- 
\begin{align*}
 \partial[\varrho]<\varrho=\atime>@{\H(\varrho)} ={}& \uniric[a](\normal,\partial[\unisymbol\atime]) - \div(\DoAaoutTensor K_{\normal}) + \trtr\zFund\aoutzFund,
\end{align*}
where we used the notation of Proposition \ref{change_of_meancurvature}. By the Codazzi equation, we obtain
\[ \partial[\varrho]<\varrho=\atime>@{\H(\varrho)} = D_{\normal}(\aouttr[\atime]\,\aoutzFund) - (\aoutdiv[\atime]\,\aoutzFund)(\normal) - \div(\aoutzFund_{\normal}) + \trtr\zFund\aoutzFund, \]
where $\aouttr[\atime]\,\aoutzFund$ is the trace of $\aoutzFund$ and the one-form $\aoutdiv[\atime\,]\aoutzFund$ is the divergence of $\aoutzFund$ -- both with respect to the metric $\aoutg[\atime]*$. By recognizing that ${-}2\,\aoutzFund_{ij}=\outg_{ij}-\schwarzg_{ij}$ and using the assumptions on $\outg*$, we see that $\vert\aoutzFund\vert\le\nicefrac C{\Hradius^{1+{\outve}}}$ and $\vert\outlevi*\aoutzFund\vert\le\nicefrac C{\Hradius^{2+{\outve}}}$. Therefore, we conclude the desired inequality with \eqref{stability_operator_on_arnu}.
\end{proof}\pagebreak[2]
By Lemma \ref{the_lowest_eigenvalues_of_the_stability_operator}, we conclude the following inequalities for $\arnu$ from the regularity of the Laplace operator:
\begin{lemma}[Regularity of the lapse functions in spacetime]\label{regularity_of_the_spacetime_lapse_function}
There are constants $C=\Cof[m][{\outve}][\oc][\c]$ and $\Cof{\Hradius_0}[m][{\outve}][\oc][\c]$ such that for any $(\atime,\Hradius)\in I_{\c}$ with $\Hradius>\Hradius_0$
\[ \Vert \arnu \Vert_{\Hk^2(\M)} \le C\Hradius^{2-{\outve}}, \qquad\qquad \Vert \arnu \Vert_{\Wkp^{1,\infty}(\M)} \le C\Hradius^{1-{\outve}}. \]
\end{lemma}
By definition, every evolution $\psi$ of a sphere is characterized by the lapse function ${\arnu[\psi]}$ and the shift ${\atbeta}$. In particular, this is true for the \lq movement\rq\ of the spheres, i.\,e.\ the change in time $\atime$ of the coordinate origin. By computing the $\Lp^2$ norm of $\nu_i$, we recognize that any function $f\in\Lp^2$ can be written as 
\[ f=f_{\nu} + f_{\text R}:=3\sum_i\normal_i\fint_{\M} \normal_i\, f \d\mug + f_{\text R}, \]%
where $f_{\text R}$ is $\Lp^2$-orthogonal to $\nu_i$ up to an error. This error vanishes asymptotically and is explained by the fact that the $\nu_i$ are only asymptotically $\Lp^2$-orthogonal to each other. It is intuitively clear that this $\arnu[\psi]_{\nu}$ part of the lapse function characterizes the \lq movement\rq\ of the spheres.
\begin{proposition}[Movement of the spheres by the lapse function]\label{the_moving_of_the_spheres_by_the_lapse_function}
Let $\psi:\interval-{\vartheta_0}{\vartheta_0}\times\M<\Hradius>[\tau]\to\aoutM[\tau]$ be a deformation of $\M:=\M<\Hradius>[\atime]$, i.\,e.\ $\psi(0,\cdot)=\text{id}|_{\M}$. There are constants $C=\Cof[m][{\outve}][\oc][\c]$ and $\Hradius_0=\Cof{\Hradius_0}[m][{\outve}][\oc][\c]$ neither depending on $\psi$ nor on $\atime$, such that for $\Hradius>\Hradius_0$ the Euclidean coordinate centers $\centerz<\vartheta>$ of the hypersurface $\M<\vartheta>:=\psi(\vartheta,\M)$ fulfill
\[ \vert \partial[\vartheta]@{\centerz<\vartheta>_i} - 3\fint_{\M} \normal_i\lapse[\psi] \d\mug \vert \le \frac C{\Hradius^2}\Vert\lapse[\psi]\Vert_{\Lp^2(\M)}, \]%
where $\lapse[\psi]:=\aoutg[\atime](\spartial[\vartheta]@\psi,\normal)$ is the lapse function of $\psi$ and $\normal$ denotes the outer unit normal of $\M<\Hradius>[\atime]$ with respect to the metric $\outg[\atime]$.
\end{proposition}
\begin{proof}
Let $\psi$ be a evolution of $\M=\M<\Hradius>[\atime]$. Again, we suppress the indices $\atime$ and $\Hradius$. As the Euclidean coordinate center is invariant under diffeomorphisms, we can assume that $\spartial[\vartheta]@\psi=\lapse[\psi]\,\nu$. For the desired inequality, we first approximate the derivation of the numerator ($\centerz<\vartheta>\,\volume{\M<\vartheta>}=\int_{\M<\vartheta>}\outsymbol x_i\d\mug$):
\begin{align*}\hspace{5.5em}&\hspace{-5.5em}
 \vert\partial[\vartheta]@{(\centerz<\vartheta>_i\,\volume{\M<\vartheta>})} - 3\int_{\M} \normal_i\,\lapse[\psi] \d\mug + \H\,\lapse[\psi]\, z_i \d\mug\vert \\
  ={}& \vert\int_{\M} \partial[\vartheta]@{(\outsymbol x_i\circ\psi)} - \H\,\lapse[\psi]\,\outsymbol x_i - 3 \normal_i\,\lapse[\psi] + \H\,\lapse[\psi]\, z_i \d\mug\vert \\
  ={}& \vert\int_{\M} \H\,\lapse[\psi]\,(\outsymbol x_i-\centerz_i)
			+ 2 \normal_i\,\lapse[\psi] \d\mug \vert \\
  \le{}& C\,\Vert\lapse[\psi]\Vert_{\Lp^2(\M)}\,(\Vert N-\nu\Vert_{\Lp^2(\M)}+\frac C\Hradius)
  \le C\,\Vert\lapse[\psi]\Vert_{\Lp^2(\M)}.
\end{align*}
Using the Leibniz formula, we conclude the claimed inequality by
\begin{align*}\hspace{4em}&\hspace{-4em}
 \vert\partial[\vartheta]@{(\centerz<\vartheta>_i)} - 3\fint_{\M} \normal_i\,\lapse[\psi] \d\mug\d\mug\vert \\
  \le{}& \volume{\M}^{-1}\vert\partial[\vartheta]@{(\centerz<\vartheta>_i\,\volume{\M[\vartheta]})}
			+ \centerz_i\int\H\,\lapse[\psi]\d\mug
			- 3\int_{\M} \normal_i\lapse[\psi] \d\mug\vert \\
	\le{}& \frac C{\Hradius^2}\,\Vert\lapse[\psi]\Vert_{\Lp^2(\M)}.\qedhere\smallskip
\end{align*}%
\end{proof}\pagebreak[3]
We can now combine the Lemmas \ref{regularity_of_the_radial_lapse_function} and \ref{regularity_of_the_spacetime_lapse_function} and Proposition \ref{the_moving_of_the_spheres_by_the_lapse_function} to finish the proof of Proposition \ref{leaves_are_almost_concentric}:\pagebreak[1]
\begin{proof}[Proof (of Proposition \ref{leaves_are_almost_concentric})]
By Lemmas \ref{regularity_of_the_radial_lapse_function} and \ref{regularity_of_the_spacetime_lapse_function} and Proposition \ref{the_moving_of_the_spheres_by_the_lapse_function}, there are constants $C=\Cof[m][{\outve}][\oc][\c]$ and $\Hradius_0=\Cof{\Hradius_0}[m][{\outve}][\oc]$ such that for $(\atime,\Hradius)\in I_{\c}$ with $\Hradius\ge\Hradius_0$
\[
 \vert\partial[\atime]@{\centerz<\Hradius>[\atime]}\vert \le C\,\Hradius^{1-{\outve}}, \qquad\qquad
 \vert\partial[\Hradius]@{\centerz<\Hradius>[\atime]}\vert \le \frac C{\Hradius^{\outve}}. \pagebreak[1] \]
Using the regularity of $\aoutsymbol\Phi$, this implies that $I_{\c}$ is open in $\interval*0*1\times\interval{\Hradius_0}\infty$, where $\Hradius_0$ and $\c$ are without loss of generality sufficiently large. As we already know, $I_{\c}$ is non-empty and closed in and therefore equal to $\interval*0*1\times\interval{\Hradius_0}\infty$. Per definition of $I_{\c}$, this concludes the proof.\pagebreak[3]
\end{proof}

\section{The evolution of the CMC center of mass}\label{The_evolution}
In this section, we characterize the evolution of the (abstract and coordinate) CMC center of mass under the Einstein equations. As mentioned in the introduction, the decay assumptions on the momentum density $\momenden$ of $\outM$ are not sufficient to ensure that the ADM linear momentum $\impuls$ of $(\outM,\outg*)$ is well-defined. Thus, we cannot expect that the evolution is characterized solely by mass and ADM linear momentum. We have to replace it by its approximating integrals (see Remark \ref{definition_of_linear_momentum}). Secondly, due to the weak assumptions on the momentum density, we need an additional correction term. This term is given by an integral, too.
\begin{theorem}[Evolution of the CMC leaves]\label{evolution_of_the_leaves}
Let $(\uniM,\unig*)$ be a Lorentzian spacetime and ${\outve}>0$. Let $\lbrace\outM[t]\rbrace_t$ be a $\Ck^2$-asymp\-totically Schwarz\-schildean temporal foliation of $(\uniM,\unig*)$ of order $1+{\outve}$. Let $\lbrace\M<\Hradius>[t]\rbrace$ be the foliation of $\outM[t]$ near infinity by spheres $\M<\Hradius>[t]$ of constant mean curvature $\H<\Hradius>[t]\equiv\nicefrac{{-}2}\Hradius+\nicefrac{4m}{\Hradius^2}$ with respect to the induced metric $\outg[t]*$ on $\outM[t]$. There are constants $\DoATensor C[t]=\Cof[m][\oc[t]][{\outve}]$ and $\DoATensor{\Hradius_0}[t]=\Cof{\DoATensor{\Hradius_0}[t]}[m][\oc[t]][{\outve}]$ such that the Euclidean coordinate center $\centerz<\Hradius>[t]$ of $\M<\Hradius>[t]$ fulfills for $\Hradius>\DoATensor{\Hradius_0}[t]$
\begin{align*}
 \vert \partial[t]@{\centerz<\Hradius>[t]_i} - \frac1{8\pi m}\int_{\M<\Hradius>[t]} \outH[t]\,\normal<\Hradius>[t]_i - \outzFund[t]_{ij}\,{\normal<\Hradius>[t]^j} + \Hradius\;\momenden(\normal<\Hradius>[t])\,{\normal<\Hradius>[t]_i} \d\mug \vert
  \le{}& \frac{\DoATensor C[t]}{\Hradius^{\min\{{\outve},\outdelta\}}}, \labeleq{Evolution_Euclidean_coordinate_center}
\end{align*}
where $\nu<\Hradius>[t]$ denotes the outer unit normal of $\M<\Hradius>[t]$ in $\outM[t]$, $\momenden[t]*(\cdot):=\uniric*(\outnu[t],\cdot)$ the momentum density, $\outnu[t]$ the future-pointing unit normal of $\outM[t]$ in $\uniM$ and the constant $\outdelta$ is as in Definition \ref{Ck_asymptotically_Schwarzschildean_initial_data_set}.
\end{theorem}
\begin{remark}[Comparison to the ADM linear momentum]\label{definition_of_linear_momentum}
The first two terms in \eqref{Evolution_Euclidean_coordinate_center} are motivated by the definition of the ADM linear momentum $\impuls$
\begin{align*}
 \impuls_i := \frac1{8\pi}\lim_{\rradius\to\infty} \int_{\{\rad\equiv\rradius\}} \momentum(\partial[\outsymbol x]_i,\normal[\{\rad\equiv\rradius\}]) \d\Haus^2 \quad\text{with}\quad
 \momentum* := \outmc\;\outg* - \outzFund,
\end{align*}
where $\normal[\{\rad\equiv\rradius\}]$ is the normal and $\Haus^2$ is the canonical measure of the Euclidean coordinate sphere $\{\rad\equiv\rradius\}:=\rad^{-1}(\rradius)$ -- both with respect to the Euclidean metric $\eukg$. By Stokes' theorem, we see that the ADM linear momentum is well defined if $\outdiv\momentum*=\overline J$ is small, e.\,g.\ $\vert\overline J\vert\in\Lp^1(\outM)$. Similarly, we see that under this additional assumption the Euclidean coordinate spheres $\{\rad\equiv\rradius\}$ can be replaced by the CMC spheres $\M<\Hradius>$. Thus, the difference of these terms in \eqref{Evolution_Euclidean_coordinate_center} has to be understood as a quasi-local linear momentum on $\M<\Hradius>$. The additional term is necessary due to the fact that we cannot use Stokes' theorem, as the momentum density is not assumed to be small enough.
\end{remark}
\begin{corollary}[Evolution of the leaves -- under stronger assumptions]\label{evolution_of_the_leaves__under_stronger_assumptions}
Let $(\uniM,\unig*)$ be a Lorentzian spacetime and ${\outve}>0$. Let $\lbrace\outM[t]\rbrace_t$ be a $\Ck^2$-asymp\-to\-ti\-cally Schwarz\-schildean temporal foliation of $(\uniM,\unig*)$ of order $1+{\outve}$ and $\vert\momenden[t]\vert\in\Lp^1(\outM[t])$ for all $t$. Let $\lbrace\M<\Hradius>[t]\rbrace_\Hradius$ be the foliation of $\outM[t]$ by spheres $\M<\Hradius>[t]$ of constant mean curvature $\H<\Hradius>[t]\equiv\nicefrac{{-}2}\Hradius+\nicefrac{4m}{\Hradius^2}$ with respect to the induced metric $\outg[t]*$ on $\outM[t]$. There are constants ${\DoATensor C[t]}=\Cof{\DoATensor C[t]}[m][{\oc[t]}][{\outve}]$ and $\DoATensor{\Hradius_0}[t]=\Cof{\DoATensor{\Hradius_0}[t]}[m][{\oc[t]}][{\outve}]$ such that the Euclidean coordinate center $\centerz<\Hradius>[t]$ of $\M<\Hradius>[t]$ with $\Hradius>\DoATensor{\Hradius_0}[t]$ fulfills
\begin{align*}
 \vert \partial[t]@{\centerz<\Hradius>[t]} - \frac\impuls m \vert \le{}& \frac{\DoATensor C[t]}{\Hradius^{\min\{{\outve},\outdelta\}}},
\end{align*}
where $\outdelta$ is as in Definition \ref{Ck_asymptotically_Schwarzschildean_initial_data_set}.
\end{corollary}%
\begin{corollary}[Evolution of the CMC center of mass]\label{evolution_of_the_center}
Let $(\uniM,\unig*)$ be a Lorentzian spacetime and ${\outve}>0$. Let $\lbrace\outM[t]\rbrace_t$ be a $\Ck^2$-asymp\-to\-ti\-cally Schwarz\-schildean temporal foliation of $(\uniM,\unig*)$ of order $1+{\outve}$ and $\vert\momenden[t]\vert\in\Lp^1(\outM[t])$ for all $t$. Let $\centerz<\Hradius>[t]$ be the Euclidean coordinate center of the foliation of $\outM[t]$ near infinity by spheres of constant mean curvature $\H<\Hradius>[t]\equiv\nicefrac{{-}2}\Hradius+\nicefrac{4m}{\Hradius^2}$ with respect to the induced metric $\outg[t]*$ on $\outM[t]$. If the coordinate CMC center of mass
\begin{align*}
 \outcenter[t]_{\text{CMC}} :={}& \lim_{\Hradius\to\infty}\centerz<\Hradius>[t]
\end{align*}
converges at one time $t=t_0\in I$ then it converge for all times $t$ and fulfills
\begin{align*}
 \partial[t]@{\outcenter[t]_{\text{CMC}}} = \frac\impuls m.
\end{align*}
\end{corollary}%
Both corollaries are direct implications of Theorem \ref{evolution_of_the_leaves}.\smallskip\pagebreak[2]

By Proposition \ref{the_moving_of_the_spheres_by_the_lapse_function}, $\spartial[t]@{\centerz<\Hradius>[t]_i}$ is asymp\-to\-ti\-cally characterized by the lapse function $\tnu<\Hradius>[t]$ of the evolution of $\M<\Hradius>[t]$ (in time $t$ for fixed $\Hradius$), i.\,e. for any smooth evolution $\psi:\interval-{\varrho_0}{\varrho_0}\times\M<\Hradius>[t]\to\uniM$ with $\M<\Hradius>[t+\delta t]=\psi(\delta t,\M<\Hradius>[t])$, this lapse function is defined by
\begin{align*}
 \partial[\varrho]@\psi<\varrho=0> ={}& \tnu<\Hradius>[t]\,\normal<\Hradius>[t] + \rbeta<\Hradius>[t]^\psi, \labeleq{the_lapse_function_eq}
\end{align*}
where $\rbeta<\Hradius>[t]^\psi\in\X(\M<\Hradius>[t])$ is the shift of $\psi$ on $\M<\Hradius>[t]$. By Proposition \ref{change_of_meancurvature} this lapse function $\tnu$ is characterized in the following way:
\begin{corollary}[the lapse function]\label{the_lapse_function}
If is the lapse function on $\M<\Hradius>[t]$ as in \eqref{the_lapse_function_eq} then -- omitting the $t$ and $\Hradius$ indices --
\begin{align*}
\jacobiext*\tnu ={}& \ralpha(\div\outzFund_{\nu} - \outsymbol J(\normal) - \trtr\zFund\outzFund) - D_{\normal}\ralpha\;\tr\outzFund + 2\outzFund_{\nu}(\levi*\ralpha),
	\labeleq{jacobiext_tnu}
\end{align*}
where $\jacobiext*$ denote the stability operator on $\M$.
\end{corollary}\medskip\pagebreak[3]

Now, we can continue with the proof of Theorem \ref{evolution_of_the_leaves}:
\begin{proof}[Proof (of Theorem \ref{evolution_of_the_leaves})]
By Proposition \ref{the_moving_of_the_spheres_by_the_lapse_function} and Corollary \ref{the_lapse_function}, we conclude the claim using the following Theorem \ref{inf_evolution_of_the_leaves}.
\end{proof}\pagebreak[2]

In view of Proposition \ref{the_moving_of_the_spheres_by_the_lapse_function}, identity \eqref{the_lapse_function_eq} and Corollary \ref{the_lapse_function}, the following Theorem \ref{inf_evolution_of_the_leaves} is the instantaneous version of Theorem \ref{evolution_of_the_leaves}. We use the expression \emph{instantaneous} as no \lq real derivative\rq\ is used, but only information of the (abstract) initial data set (and the lapse function $\ralpha$).\footnote{Note that the lapse function $\ralpha$ does in fact characterize the \lq infinitesimal evolution\rq\ of spacelike hypersurface under the Einstein equations in a Lorentzian manifold.}
\begin{theorem}[Instantaneous evolution of the leaves]\label{inf_evolution_of_the_leaves}
Let $\ve>0$	 and $(\outM,\outg*,\outsymbol x,\outzFund,\enden,\outsymbol J,\ralpha)$ be an arbitrary $\Ck^2$-asymp\-to\-ti\-cal\-ly Schwarz\-schildean (abstract) initial data set of order $1+{\outve}$ and let $\lbrace\M<\Hradius>\rbrace_\Hradius$ be its foliation near infinity by constant mean curvature surfaces. There are constants $C=\Cof[m][\outve][\oc]$ and $\Hradius_0=\Cof{\Hradius_0}[m][\outve][\oc]$ such that the lapse function $\tnu<\Hradius>$ of the leaf $\M<\Hradius>$ with mean curvature $\H\equiv\nicefrac{-2}\Hradius+\nicefrac{4m}{\Hradius^2}$ and $\Hradius\ge\Hradius_0$ fulfills
\begin{align*}
 \vert3\fint_{\M} \normal<\Hradius>_i\,\tnu<\Hradius> \d\mug - \frac1{8\pi m} \int \momentum(\normal,\outsymbol e_i) + \Hradius\cdot\normal_i\cdot\outsymbol J(\normal) \d\mug\vert \le{}& \frac C{\Hradius^{\min\{{\outve},\outdelta\}}},
\end{align*}
where $\outdelta$ is as in Definition \ref{Ck_asymptotically_Schwarzschildean_initial_data_set} and $\tnu<\Hradius>$ is characterized by -- omitting the index $\Hradius$ --
\begin{align*}
 \jacobiext*\tnu ={}& \ralpha(\div\outzFund_{\normal} - \outsymbol J(\normal) - \trtr\zFund\outzFund) - D_{\normal}\ralpha\;\tr\outzFund + 2\outzFund_{\normal}(\levi*\ralpha),
\end{align*}
where $\jacobiext*$ denote the stability operator on $\M$.
\end{theorem}
\begin{proof}
Omit the index $\Hradius$. Let $\{f_i\}$ be an $\Lp^2$-orthonormal set of eigenfunctions of the stability operator $\jacobiext*$ corresponding to the eigenvalues $\lambda_i$ with $\vert\lambda_i\vert\le\nicefrac1{\Hradius^2}$. By characterization of the lapse function, we see
\begin{align*}
 \int_{\M} f_i\tnu \d\mug
	= \int\frac{f_i}{\lambda_i}(\vphantom{\frac\cdot\cdot}\ralpha(\div\outzFund_{\normal} - \outsymbol J(\normal) - \trtr\zFund\outzFund) - D_{\normal}\ralpha\;\tr\outzFund + 2\outzFund_{\normal}(\levi*\ralpha)) \d\mug.
\end{align*}
By the decay assumption on $\ralpha$ and $\outzFund*$ and the fact that $\M$ is almost asymptotically concentric by Proposition \ref{leaves_are_almost_concentric}, we conclude
\begin{align*}
 \vert D_{\normal}\ralpha\;\tr\outzFund + 2\outzFund_{\normal}(\levi*\ralpha)\vert \le{}& \frac C{\Hradius^{3+{\outve}}}, &
 \vert \ralpha(\div\outzFund_{\normal} - \outsymbol J(\normal)) - (\div\outzFund_{\normal} - \outsymbol J(\normal)) \vert \le{}& \frac C{\Hradius^{3+{\outve}}}.
\end{align*}
Considering additionally the decay \eqref{estimates_for_k} of $\zFund$, we furthermore get
\begin{align*}
 \vert\alpha\;\trtr\zFund\outzFund-\frac\H2\troutzFund\vert \le{}& \frac C{\Hradius^{3+\min\{{\outve},\outdelta\}}}.
\end{align*}
The only remaining term that we need to understand is
{\begin{equation*}\hspace{-1.1em}
 \int f_i\,(\outsymbol J(\normal) + \frac{\H\,\troutzFund}2 - \div\outzFund_{\normal}) \d\mug
	= \int \outzFund(\normal,\levi*f_i) + f_i\,(\outsymbol J(\normal) + \frac{\H\,\troutzFund}2) \d\mug.
	\labeleq{inf_evolution_of_the_leaves_left_1}\hspace{-1em}
\end{equation*}
}

To obtain the desired result, we now have to replace $f_i$ by the components of the normal. By definition of the stability operator and the assumed decay of $\outg*-\schwarzg$, we see
\[ \Vert\laplace f_i + (\frac2{\Hradius^2}-\frac{4m}{\Hradius^3}) f_i\Vert_{\Lp^2(\M)} \le \frac C{\Hradius^{2+{\outve}}}. \]
Thus, the $f_i$ are comparable to the first three non-constant $\Lp^2$-orthonormal eigenfunctions $g_i$ of the Laplace -- in particular the linear span $\lin\{f_i\}_i$ is three-dimen\-sional. In Euclidean space, the components $N_i$ of the normal to the sphere $\lbrace\rad\equiv\rradius\rbrace$ are eigenfunctions of the Laplace. Using the inequality \eqref{estimates_on_the_normal} bounding $\normal-N$, we thus conclude
\[ \Vert \normal_i - \sum_j(\int g_j\cdot\normal_i\d\mug)\; g_j\Vert_{\Hk^1(\M)} \le \frac C{\Hradius^{\outve}}. \]
Thus, there is an isomorphism $T:\lin\{\normal_i\}_{i=1}^3\to\lin\{f_i\}_{i=1}^3$ with 
\[ \Vert T-\id_{\Lp^2} \Vert_{\mathcal L(\lin\{\normal_i\}_{i=1}^3;\Lp^2(\M))} +\Vert T^{-1}-\id_{\Lp^2} \Vert_{\mathcal L(\lin\{f_i\}_{i=1}^3;\Lp^2(\M))} \le \frac C{\Hradius^{\min\{{\outve},\outdelta\}}}, \]
where the norm on the left hand side denotes the operator norm between these subspaces of $\Lp^2$ and $\Lp^2$. By Lemma \ref{the_lowest_eigenvalues_of_the_stability_operator} and \eqref{inf_evolution_of_the_leaves_left_1}, we conclude 
\begin{align*}
 \vert \int_{\M} \normal_i\,\tnu \d\mug + \frac{\Hradius^3}{6m}\int \outzFund(\normal,\levi*\normal_i) + \normal_i\,(\outsymbol J(\normal) - \frac\troutzFund\Hradius) \d\mug\vert
	\le{}& C\Hradius^{1-\min\{{\outve},\outdelta\}}\,\Vert \normal_i\Vert_{\Lp^2(\M)}. \labeleq{inf_evolution_of_the_leaves_left_2}
\end{align*}
We compare $\normal_i$ with $N_i$ and recall $\levi[e]*N_i=\nicefrac1\Hradius\;(e_i-N_i)$ on the Euclidean sphere $\lbrace\rad\equiv\rradius\rbrace$ with respect to the Euclidean metric $\eukg$. Using the constant given by $\H$, the integral asymptotically simplifies to
\begin{align*}\hspace{-2em}
 \frac{{-}\Hradius^3}{6m}\int (\frac{\outzFund(\normal,\outsymbol e_i)}\Hradius + \normal_i\,(\outsymbol J(\normal) - \frac\outmc\Hradius)) \d\mug
	={}& \frac{\Hradius^2}{6m} \int (\momentum(\normal,\outsymbol e_i) + \Hradius\,\normal_i\;\outsymbol J(\normal)) \d\mug, \hspace{-2em}\labeleq{inf_evolution_of_the_leaves_left_3}
\end{align*}
i.\,e.
\begin{align*}
 \vert3\fint_{\M} \normal_i\,\tnu \d\mug - \frac1{8\pi m} \int (\momentum(\normal,\outsymbol e_i) + \Hradius\,\normal_i\;\outsymbol J(\normal)) \d\mug\vert
   \le{}& \frac C{\Hradius^{\min\{{\outve},\outdelta\}}}.\qedhere
\end{align*}
\end{proof}\pagebreak[3]

\section{Calculating the Euclidean coordinate center of a CMC leaf and comparing the coordinate CMC and ADM center of mass}
\label{Calculating_the_Euclidean_coordinate_center_of_a_CMC_leaf_and___}
Now, we use our new understanding of the time evolution of the CMC spheres and the Construction \ref{Artificial_Lorentzian_manifolds} to (asymptotically) calculate the Euclidean coordinate center of a leaf of the CMC foliation.
\begin{corollary}[Euclidean coordinate center of a leaf]\label{euclidean_coordinate_center_of_a_leave}
Let $(\outM,\outg*)$ be a $\Ck^2$-asymptotically Schwarzschildean three-dimensional Riemannian manifold of order $1+{\outve}$ with ${\outve}>0$. Further denote by $\lbrace\M<\Hradius>\rbrace_\Hradius$ the foliation of $\outM$ near infinity by spheres of constant mean curvature $\H<\Hradius>\equiv\nicefrac{{-}2}\Hradius+\nicefrac{4m}{\Hradius^2}$. There are constants $C=\Cof[m][{\outve}][\oc]$ and $\Hradius_0=\Cof{\Hradius_0}[m][{\outve}][\oc]$ such that the Euclidean coordinate center $\centerz<\Hradius>$ of any leaf $\M<\Hradius>$ with $\Hradius>\Hradius_0$ fulfills -- omitting the index $\Hradius$ --
\begin{align*}
 \vert \centerz_i - \frac1{16\pi m} \int_{\sphere^2_\Hradius(0)}\sum_{j=1}^3(
																					\outsymbol x_i(\partial[\outsymbol x]_j@{\outg_{jk}}-\partial[\outsymbol x]_k@{\outg_{jj}})
																				- (\outg_{ij}\,\frac{\outsymbol x^j}\rad - \outg_{jj}\,\frac{\outsymbol x_i}\rad))\d\Haus^2 \vert \le{}& \frac C{\Hradius^{\outve}},
\end{align*}
where $\Haus^2$ is the canonical measure of the Euclidean coordinate sphere $\sphere^2_\Hradius(0):=\rad^{-1}(\Hradius)$.
\end{corollary}
\begin{proof}
We use the artificial Lorentzian manifold $(\uniM,\unig[a])$ as in Construction \ref{Artificial_Lorentzian_manifolds}. We conclude that the exterior curvature of any artificial time-slice $\aoutM[\atime]:=\{\atime\}\times\outM$ is given by ${\aoutzFund[\atime]*}=2(\schwarzg-\outg*)$ and the induced metric is $\aoutg[\atime]*$. In particular, we can calculate the Euclidean coordinate centers $\centerz<\Hradius>=\centerz<\Hradius>[1]$ of the leaves $\M<\Hradius>=\M<\Hradius>[1]$ of the CMC foliation $\lbrace\M<\Hradius>\rbrace_\Hradius$ near infinity of $\outM[1]$ with respect to the \lq real\rq\ metric $\outg*=\outg[1]*$, by looking at Theorem \ref{inf_evolution_of_the_leaves}. Using the fact that in coordinates $\M<\Hradius>[\atime]\approx\mathcal S^2_\Hradius(\centerz<\Hradius>[\atime])$, we therefore obtain the (asymptotic) ordinary differential equations
\begin{align*}
 \centerz<\Hradius>[0\,] ={}& \boldsymbol 0, &
 \vert\partial[\atime]@{\centerz<\Hradius>[\atime]} - \frac1{8\pi m} \int_{\mathcal S^2_\Hradius(\centerz<\Hradius>[\atime])} (\amomentum[\atime\,](\nu,\outsymbol e_i) + \Hradius\,\normal_i\,\outsymbol J(\normal))\d\mug \vert \le{}& \frac C{\Hradius^{\outve}}. \labeleq{artificial_mfld_ode_zentrum}
\end{align*}
Furthermore by Proposition \ref{leaves_are_almost_concentric}, we see
\begin{align*}
 \vert\int_{\mathcal S^2_\Hradius(\centerz<\Hradius>[\atime])} (\amomentum[\atime](\nu,e_i) + \Hradius\,\normal_i\,\outsymbol J(\normal))\d\mug
	- \int_{\mathcal S^2_\Hradius(\boldsymbol 0)} (\amomentum[\atime](\normal,e_i) + \Hradius\,\normal_i\,\outsymbol J(\normal))\d\Haus^2\vert \le{}& \frac C{\Hradius^{\outve}}, \labeleq{artificial_mfld_ode_zentrum_ineq}
\end{align*}
where $\Haus^2$ is the canonical measure of the Euclidean coordinate sphere. By taking into account the assumptions on $\outg*$, we conclude the result.
\end{proof}\pagebreak[3]

\noindent Let us now quickly recall the definition of the ADM center of mass.
\begin{definition}[ADM center of mass]
Let $(\outM,\outg*)$ be a $\Ck^2$-asymptotically Schwarzschildean three-dimensional Riemannian manifold of order $1+{\outve}$ with ${\outve}>0$. The coordinate ADM center of mass $\centerz_{\text{ADM}}$ is defined to be
\begin{align*}
 (\outcenter_{\text{ADM}})_i :={}& \frac1{16\pi m} \lim_{\rradius\to\infty} \int_{\sphere^2_\rradius(0)}\sum_{j=1}^3(
																					\outsymbol x_i(\partial[\outsymbol x]_j@{\outg_{jk}}-\partial[\outsymbol x]_k@{\outg_{jj}})
																				- (\outg_{ij}\,\frac{\outsymbol x^j}\rradius - \outg_{jj}\,\frac{\outsymbol x_i}\rradius))\d\Haus^2,
\end{align*}
where indices are lowered with the Euclidean metric and $\Haus^2$ is the canonical measure on the sphere.
\end{definition}\pagebreak[1]
Using Corollary \ref{euclidean_coordinate_center_of_a_leave} for the limit $\Hradius\to\infty$, the equality of the coordinate CMC center of mass and the coordinate ADM center of mass is a direct implication of Corollary \ref{euclidean_coordinate_center_of_a_leave}.
\begin{corollary}[Euclidean coordinate center of the foliation]
Let $(\outM,\outg*)$ be a $\Ck^2$-asymptotically Schwarzschildean three-dimensional Riemannian manifold of order $1+{\outve}$ with ${\outve}>0$. If either the coordinate ADM center $\outcenter_{\text{ADM}}$ or the coordinate CMC center $\outcenter_{\text{CMC}}$ converge then so does the other one and they coincide.
\end{corollary}
\vfill
\bibliography{bib}

\begin{thebibliography}{CWY13}

\bibitem[ADM61]{arnowitt1961coordinate}
Richard Arnowitt, Stanley Deser, and Charles~W Misner.
\newblock Coordinate invariance and energy expressions in general relativity.
\newblock {\em Physical Review}, 122(3):997, 1961.

\bibitem[BE13]{brendle2013large}
Simon Brendle and Michael Eichmair.
\newblock Large outlying stable constant mean curvature spheres in initial data
  sets.
\newblock 2013.
\newblock
  \href{http://arxiv.org/abs/1303.3545v2}{arXiv:\linebreak[1]1303.3545v2}.

\bibitem[Ced12]{Cederbaum_newtonian_limit}
Carla Cederbaum.
\newblock {\em The Newtonian Limit of Geometrostatics}.
\newblock PhD thesis, 2012.
\newblock
  \href{http://arxiv.org/abs/1201.5433v1}{arXiv:\linebreak[1]1201.5433v1}.

\bibitem[CN13]{cederbaumnerz2013_examples}
Carla Cederbaum and Christopher Nerz.
\newblock Explicit {R}iemannian manifolds with unexpectedly behaving center of
  mass.
\newblock {\em accepted at Ann. Henri Poincar\'e 2014}, 2013.
\newblock \href{http://arxiv.org/abs/1312.6391}{arXiv:\linebreak[1]1312.6391}.

\bibitem[CS06]{corvino2006asymptotics}
Justin Corvino and Richard~M Schoen.
\newblock On the asymptotics for the vacuum einstein constraint equations.
\newblock {\em Journal of Differential Geometry}, 73(2):185--217, 2006.

\bibitem[CWY13]{chen_wang_yau__Quasilocal_angular_momentum_and_center_of_mass_in_gr}
Po-Ning Chen, Mu-Tao Wang, and Shing-Tung Yau.
\newblock Quasilocal angular momentum and center of mass in general relativity.
\newblock 2013.
\newblock
  \href{http://arxiv.org/abs/1312.0990v1}{arXiv:\linebreak[1]1312.0990v1}.

\bibitem[CY88]{christodoulou71some}
Demetrios Christodoulou and Shing~Tung Yau.
\newblock Some remarks on the quasi-local mass.
\newblock In {\em Mathematics and general relativity ({S}anta {C}ruz, {CA},
  1986)}, volume~71 of {\em Contemp. Math.}, pages 9--14. Amer. Math. Soc.,
  Providence, RI, 1988.

\bibitem[DLM05]{DeLellisMueller_OptimalRigidityEstimates}
Camillo De~Lellis and Stefan M{\"u}ller.
\newblock {O}ptimal rigidity estimates for nearly umbilical surfaces.
\newblock {\em {J}ournal of {D}ifferential {G}eometry}, 69(1):075--110, 2005.

\bibitem[EM12]{metzger_eichmair_2012_unique}
Michael Eichmair and Jan Metzger.
\newblock Unique isoperimetric foliations of asymptotically flat manifolds in
  all dimensions.
\newblock {\em Inventiones mathematicae}, pages 1--40, 2012.

\bibitem[Hua10]{Lan_Hsuan_Huang__Foliations_by_Stable_Spheres_with_Constant_Mean_Curvature}
Lan-Hsuan Huang.
\newblock Foliations by {S}table {S}pheres with {C}onstant {M}ean {C}urvature
  for isolated systems with general asymptotics.
\newblock {\em Communications in Mathematical Physics}, 300(2):331--373, 2010.

\bibitem[HY96]{huisken_yau_foliation}
Gerhard Huisken and Shing-Tung Yau.
\newblock Definition of center of mass for isolated physical systems and unique
  foliations by stable spheres with constant mean curvature.
\newblock {\em Invent. Math.}, 124:281--311, 1996.

\bibitem[LMS11]{lamm2011foliationsbywillmore}
Tobias Lamm, Jan Metzger, and Felix Schulze.
\newblock Foliations of asymptotically flat manifolds by surfaces of willmore
  type.
\newblock {\em Mathematische Annalen}, 350(1):1--78, 2011.

\bibitem[Met07]{metzger2007foliations}
Jan Metzger.
\newblock Foliations of asymptotically flat 3-manifolds by 2-surfaces of
  prescribed mean curvature.
\newblock {\em Journal of Differential Geometry}, 77(2):201--236, 2007.

\bibitem[Ner14]{nerz2014CMCfoliation}
Christopher Nerz.
\newblock Foliations by stable spheres with constant mean curvature for
  isolated systems without symmetry.
\newblock {\em pre-print}, 2014.
\newblock
  \href{http://de.arxiv.org/pdf/1408.0752}{arXiv:\linebreak[1]1408.0752v1}.

\bibitem[QT07]{qing2007uniqueness}
Jie Qing and Gang Tian.
\newblock On the uniqueness of the foliation of spheres of constant mean
  curvature in asymptotically flat 3-manifolds.
\newblock {\em Journal of the American Mathematical Society}, 20(4):1091--1110,
  2007.

\end{thebibliography}
\bibliographystyle{alpha}\vfill
\end{document}